%% file: complex-soumis-02-arxiv.tex
\documentclass[a4paper]{article}

\input{macros-arxiv.tex}

\newcommand{\A}{\ron A}

\newcommand{\pe}{\mathrm{pe}}

\newcommand{\Ho}{{\mathrm{H}}}   

\title{Complexity as a homeomorphism invariant for tiling spaces}

\date{January 2014}

\author{Antoine Julien\footnote{Part of the early work on this topic was realized during a previous post-doc at the University of Victoria, under the supervision of Ian Putnam, partially funded by the Pacific Institute for the Mathematical Sciences.} \\
\small{Department of Mathematical Sciences} \\
\small{Norwegian University of Science and Technology (NTNU)} \\
\small{N-7034 Trondheim, Norway}\\
\small{antoine.julien@math.ntnu.no} }

\begin{document}

\maketitle

\begin{abstract}
 It is proved that whenever two aperiodic repetitive tilings with finite local complexity have homeomorphic tiling spaces, their associated complexity functions are asymptotically equivalent in a certain sense (which implies, if the complexity is polynomial, that the exponent of the leading term is preserved by homeomorphism).
 This theorem can be reworded in terms of $d$-dimensional infinite words: if two $\Z^d$-subshifts (with the same conditions as above) are flow equivalent, their complexity functions are equivalent.
 An analogue theorem is stated for the repetitivity function, which is a quantitative measure of the recurrence of orbits in the tiling space.
 How this result relates to the theory of tilings deformations is outlined in the last part.
\end{abstract}

\tableofcontents

\section{Introduction}

 This article is concerned with the study of aperiodic tilings. These tilings, such as the well-known Penrose tilings, provide good models for \emph{quasicrystals} in physics.
 The tilings in which we are interested display two \emph{a priori} antagonistic properties: a highly ordered structure (in the form of uniform repetition of patches), and aperiodicity.

 A tiling is a covering of the plane by geometric shapes (tiles), with no holes and no overlap.
 In this paper, tilings are assumed to be made using finitely many tiles types up to translation.
 Furthermore, it is assumed that any two tiles can be fitted together in only finitely many ways---for example it could be assumed that tiles meet full-face to full-face in the case of polytopal tiles. This last condition is known as \emph{finite local complexity}.
 Finally, tilings we are dealing with are assumed to be \emph{repetitive} in the sense that any finite patch of a given tiling repeats within bounded distance of any point in the tiling (these definitions are given with more precision and quantifiers in Section~\ref{sec:prelim}).

 An aperiodic tiling with the above properties has a topological compact space associated with it. It consists of all tilings which are indistinguishable from it at a local scale, and it supports a $\R^d$-action (for tilings of dimension $d$), induced by translations.
 It is possible to gain a better understanding of such tiling spaces by using an analogy with subshifts: a multi-dimensional infinite word $w \in \ron A^{\Z^d}$ has a subshift associated with it. It is the closure in $\ron A^{\Z^d}$ of the $\Z^d$-orbit of $w$ under the shift (or translation) map.
 The word $w$ can also be interpreted as a tiling by cubes, the colours of which are indexed by $\ron A$. The tiling space of this tiling by cubes then corresponds exactly to the \emph{suspension} of the subshift (which for $d=1$ is also known as the mapping torus of the shift map).
 This suspension contains the subshift as a closed subset, and supports an $\R^d$-action which extends the $\Z^d$-action on the subshift.
 
 We see here that if a subshift is minimal (which corresponds exactly to the condition cited above on repetition of patches), all elements in the subshifts have the same \emph{language}, \ie the same set of finite subwords.
 In the same way, if a tiling space is minimal, it make more sense to study the space rather than to particularise an arbitrarily chosen tiling.
 
 This philosophy of studying the space rather than the tiling (or the subshift rather than the word) is far from new, and has been quite successful.
 One can specifically cite the gap-labeling theorems, which relate the topological invariants of the space on the one hand (the ordered $K_0$-group), and the gaps in the spectrum of a certain Schr\"odinger operator with aperiodic potential on the other hand.
 Any attempt to give complete references for this problem would be unfair; we will therefore give a very incomplete list of references in the form of the papers~\cite{BHZ00,Kel95,BBG06} and their references.

 A natural question, however, which results directly from this ``space-over-tiling'' approach is the following: \emph{whenever a relation is established between two tiling spaces, to which extent does it translate back in terms of tilings?}

 In this paper, we start with two tiling spaces which are equivalent in the realm of topological spaces, \ie homeomorphic. We then investigate what are the consequences on the underlying tilings.
 The main results states that whenever two tiling spaces are homeomorphic, their respective complexity functions, as well as their repetitivity functions are equivalent in some sense.
 
 Given a tiling (or a word), the complexity function is a map $r \mapsto p(r)$, where $p(r)$ counts the number of distinct patches of size $r$, up to translation.
 This function was first studied in the framework of symbolic dynamics (one-dimensional subshifts).
 In their seminal paper, Morse and Hedlund~\cite{MH38} define $p(n)$ as the number of subwords of size~$n$ in a given bi-infinite word.
 This function provides a good measure of order and disorder in a word: if $p(n)$ doesn't grow at least linearly, then the word is periodic.
 This result was generalized outside of the symbolic setting, in higher dimension: if the complexity function associated with a tiling doesn't grow at least linarly, then the tiling is completely periodic \ie it has $d$ independent periods.
 One can refer to Lagarias--Pleasants~\cite{LP03} for a proof of this result in the setting of Delaunay sets of~$\R^d$.

 The first main theorem of this article describes how the complexity function is preserved whenever two tiling spaces are homeomorphic.
 \begin{thm*}[(Theorem~\ref{thm:main})]
  If $h:\Omega \lra \Omega'$ is a homeomorphism between two tiling spaces of aperiodic and repetitive tilings with finite local complexity, then the associated complexity functions $p$ and $p'$ satisfy the following inequalities:
  \[
    c p (mr) \leq p' (r) \leq C p (Mr),
  \]
  for some constants $c,C,m,M>0$ and all $r$ big enough.
 \end{thm*}
 
 A few remarks:
 \begin{enumerate}
  \item If the functions $p$ and $p'$ are at least linear (which is the case for non-periodic tilings, by the Morse--Hedlund theorem), then $c$ and $C$ can be taken equal to $1$ (up to adjusting the values of $m$ and $M$);
  \item On the opposite, if the functions $p$ and $p'$ grow at most polynomially, then $m$ and $M$ can be chosen equal to $1$ (up to adjusting $c$ and $C$). In particular, if $p$ grows like a polynomial, then so does $p'$, and the exponent of the leading term is the same.
 \end{enumerate}

 \begin{cor*}
  Let $T$ be an aperiodic, repetitive tiling with finite local complexity. Assume that its complexity function grows asymptotically like $r^\alpha$, up to a multiplicative constant. Then the exponent $\alpha$ is a topological invariant of the tiling space associated with~$T$.
 \end{cor*}

 The repetitivity function was also introduced by Morse and Hedlund~\cite{MH38} in the setting of one-dimensional symbolic dynamics.
 The repetitivity function for a tiling $T$ is defined as follows: for $r> 0$, $\ron R(r)$ is the infimum of all numbers $c$ which satisfy that any patch of radius $c$ in $T$ contains a copy of all patches of radius~$r$ which appear somewhere in~$T$.
 It is well defined and finite for all $r$ exactly when the tiling $T$ is \emph{repetitive} \ie satisfies the aforementioned condition on repetition of patches.
 It is a quantitative measure of the patch-repetition property.

 It can also be seen as a measure of order and disorder in a tiling dynamical system, which is distinct from the complexity, though related, and is sometimes finer.
 For example, if $\ron R(r)$ doesn't grow at least linearly, the tiling has $d$ independent periods~\cite{LP03}.
 Low repetitivity implies low complexity: Lenz~\cite{Len04} proved that if $\ron R$ is bounded above by a linear function, then the complexity function of the tiling is bounded above by $C r^d$ for some constant~$C$.
 On the other hand, it is possible to build some tilings which have low complexity (sub-linear in dimension~$1$), but have a repetitivity function which grows much faster than linearly.

 \begin{thm*}[(Theorem~\ref{thm:repet})]
  If $h:\Omega \lra \Omega'$ is a homeomorphism between two tiling spaces of aperiodic and repetitive tilings with finite local complexity, then the associated repetitivity functions $\ron R$ and $\ron R'$ satisfy
  \[
    c \ron R (mr) \leq \ron R' (r) \leq C \ron R (Mr)
  \]
  for some constants $c,C,m,M>0$ and all $r$ big enough.
 \end{thm*}

 In the literature, it is more common to make the qualitative distinction between linearly repetitive tiling spaces and those which are not.
 A tiling space is linearly repetitive (or linearly recurrent) if its repetitivity function satisfies $\ron R(r) \leq \lambda r$ for some constant~$\lambda$.
 Linearly repetitive tiling spaces enjoy nice properties (for example they are uniquely ergodic), and many of the actual models for quasicrystals are linearly repetitive.
 This property is preserved by homeomorphism
 
 \begin{cor*}
  If two tiling spaces satisfying the assumption of the previous theorem are homeomorphic, then either they are both linearly repetitive, or none is.
 \end{cor*}

 The results above show that there exist a significant relationship between the topology and the dynamics of a tiling space on the one hand, and its complexity and repetitivity function on the other hand.
 This is \emph{a priori} surprising, since the complexity and repetitivity do not stand out naturally as being topological objects.
 The repetitivity function clearly is strongly related to the dynamics on the tiling space, as it can be interpreted in terms of maximal size of return vectors on a transversal, or as a measure of the recurrence of orbits in the dynamical system.

 The complexity function can be related with the dynamics of the action of $\R^d$ on the tiling space 
 Given a tiling space, it is possible to define the \emph{patch-counting entropy} (or configurational entropy) by the formula:
 \[
  H_\mathrm{pc} (\Omega) = \limsup_{r \ra +\infty} \frac{\log p(r)}{r^d}.
 \]
 It was proved to be equal to the topological entropy of $(\Omega,\R^d)$ when $\Omega$ is a tiling space with finite local complexity (Baake--Lenz--Richard~\cite{BLR07}).
 In particular, using the results above, it is possible to state the following result.
 \begin{prop*}
  Let $\Omega$, $\Omega'$ be two aperiodic, repetitive, FLC tiling spaces which are homeomorphic. Then $(\Omega,\R^d)$ has positive topological entropy if and only if $(\Omega',\R^d)$ does.
 \end{prop*}
 Note that when the entropy is not zero, it is \emph{a priori} not preserved.

 When the entropy is zero, however, the complexity function appears rather to be a \emph{metric} object---when the space is endowed with a commonly used distance (sometimes called the combinatorial metric).
 It had been noticed before~\cite{JS11,Jul-phd}, and it is discussed in more details in Section~\ref{ssec:statement}.
 The consequences of the results proved in this paper can be summed up by the following statement.
 \begin{prop*}[(Corollary~\ref{cor:dimbox})]
  A homeomorphism between two aperiodic, repetitive, FLC tiling spaces preserves the box-counting dimension of their transversals, when they are endowed with the combinatorial distance.
 \end{prop*}

 The relationships between complexity on the one hand, and the dynamics on the other hand, had already been investigated to some extent.
 For example, for subshifts, it is known that a topological conjugacy is a so-called ``sliding block code'' (which is referred to as ``local derivation'' in our context, see Definition~\ref{def:loc-deriv}), and asymptotic properties of the complexity function are unchanged by an invertible recoding.
 For tiling spaces, there is no longer an equivalence between topological conjugacy and local derivation. Frank and Sadun~\cite{FS12} proved in a more general setting (including even infinite local complexity), that the asymptotic behaviour of the complexity function is preserved by a topological conjugacy of the tiling spaces.

 The reason why a mere homeomorphism---and not a conjugacy---has consequences beyond the topology has to do with the special structure of the space.
 Under the hypotheses made in this paper, a homeomorphism automatically sends orbits to orbits, in a much controlled way.
 \begin{prop*}[(Theorem~\ref{thm:hT} and Corollary~\ref{cor:lipschitz})]
  A homeomorphism $h$ between two aperiodic, repetitive tiling spaces with finite local complexity is an orbit equivalence. Furthermore, for all $T$, the cocycle $h_T$ defined by $h(T-x) = h(T)-h_T(x)$ is a homeomorphism from $\R^d$ onto itself, and satisfies the following inequality:
  \[
   \nr{h_T(x)} \leq M \nr{x} + C
  \]
  for some constants $M,C$ which do not depend on~$T$.
 \end{prop*}
 This result, together with an approximation lemma (Lemma~\ref{lem:redressement}), is the key ingredient to extend Frank and Sadun's ideas and prove the main results above.

 The methods for proving Theorem~\ref{thm:main} (namely Lemma~\ref{lem:quasi-lipschitz}) can also be related with results presented in~\cite[Chapitre~6]{Jul-phd}.
 It was shown that the image of a tiling by a small deformation has the same complexity function as the original tiling (up to equivalence).
 A small deformation induces a special kind of homeomorphism between the original tiling space and the space of the deformed tiling.
 This result can therefore be seen as a particular situation of the case treated in the present paper.
 
 The result of~\cite{Jul-phd} was in particular applied to a family of deformations introduced by Sadun and Williams.
 In~\cite{SW03}, it is proved that a tiling space is homeomorphic to the suspension of a subshift; in~\cite{Jul-phd}, it is proved that this homeomorphism preserves the complexity.

 It turns out not to be a coincidence that methods which applied for the study of complexity of deformed tilings can be used here.
 In a last section, some of the links between homeomorphisms and deformations are sketched. Given $h: \Omega \ra \Omega'$, it can be perturbed to define a deformation.
 It also defines an element in the mixed cohomology group $\Ho^1_\mathrm{m}(\Omega, \R^d)$ defined by Kellendonk~\cite{Kel08}.
 This group is a good candidate for describing homeomorphisms between tiling spaces, up to topological conjugacy.
 A more thorough investigation of the questions raised in the last section of this paper will be the subject of a future article.

\subsection*{Acknowledgements}
I wish to thank Ian Putnam, Johannes Kellendonk, Lorenzo Sadun and Christian Skau for useful discussions. I also wish to thank Michael Baake who gave me the opportunity to present an early version of this work at the University of Bielefeld.
Part of the early work on this topic was realized during a previous post-doc at the University of Victoria, under the supervision of Ian Putnam, partially funded by the Pacific Institute for the Mathematical Sciences.

\section{Preliminaries}\label{sec:prelim}

\subsection{Tilings and tiling spaces}\label{ssec:tilings}

In our discussion, a \emph{tile} is a closed compact set of the Euclidean space $\R^d$, which is homeomorphic to a closed ball, and is the closure of its interior.
A \emph{prototile} is the equivalence class under translation of a tile.
Note that a tile may be decorated (one can think of blue and red squares, for example). Formally, it would be a pair $(t,l)$, where $t$ is the tile itself (a subset of $\R^d$) and $l$ is a label. This notation will remain implicit, and a tile will just be denoted by~$t$.

Let $\A$ be a finite collection of prototiles, which is fixed from this point on.
A tiling with prototiles in $\A$ is a collection of tiles $T = \{t_i\}_{i \in I}$, such that for all~$i$, the translation class of $t_i$ belongs to $\A$. Furthermore, the union of the tiles is all of $\R^d$, and the tiles may only intersect on their boundaries: $t_i \cap t_j$ has no interior, unless $i=j$.

Given a tile $t \subset \R^d$, and $x \in \R^d$, the tile $t-x$ is defined as the translate of $t$ by $x$.
In a straightforward way, $T-x$ is defined as $\{t_i - x\}_{i \in I}$.
It is a tiling, and is in general different from $T$ (in our discussion, tilings are \emph{not} being identified when they are translates of each other).

A tiling $T$ is \emph{periodic} if there exists $x \in \R^d \setminus \{0\}$, such that $T-x = T$. It is called \emph{non-periodic} or aperiodic otherwise.

Given a non-periodic tiling, it makes sense to study the set of all tilings which cannot be distinguished from it at the local scale. This is the motivation for constructing the tiling space, in the same way that a subshift can be associated with a word.

\begin{defin}
 Given $T$ a tiling of $\R^d$, made from a set of prototiles $\A$, the \emph{hull} of $T$ is the following closure:
 \[
  \Omega_T = \overline{ \{T-x \tq x \in \R^d \} }.
 \]
 The closure is taken in the set $\Omega_\text{all}$ of all tilings made from tiles in $\A$, for the topology of the distance defined by:
 \begin{multline*}
  D(T_1, T_2) \leq \eps \text{ if } \exists x_1,x_2 \in \R^d, \text{ with } \nr{x_i} \leq \eps, \\
  \text{such that }  (T_1 - x_1) \text{ agrees with } (T_2 - x_2) \text{ on } B(0, 1/\eps ),
 \end{multline*}
 and $D$ is bounded above by $1/\sqrt{2}$.
\end{defin}

Note that all elements in $\Omega_T$ are tilings themselves, and therefore there is an action of $\R^d$ by translation on $\Omega_T$.
This action is continuous for the topology defined above.

A tiling $T$ is said to have \emph{finite local complexity} (or FLC) if for all $R > 0$ there is a finite number (up to translation) of patches of the form $T \cap B(x,R)$, for $x \in \R^d$. Here, $T \cap B(x,R)$ stands for the patch of all tiles which intersect the ball $B(x,R)$.

A tiling $T$ is said to be \emph{repetitive} if for all finite patch $P \subset T$, there exists $R > 0$ such that a translate of $P$ appears in $T \cap B(x,R)$, for any $x \in \R^d$.

The following proposition is well known.
\begin{prop}
 Let $T$ be a non-periodic and repetitive tiling, with finite local complexity. Then the tiling space $\Omega_T$ is a compact space, and the action by translation is minimal.
 When there is no risk of ambiguity, the tiling space is just written $\Omega$.
\end{prop}

Note that all these definitions--- which hold for tilings---will also hold for other patterns, as long as it is possible to define what a patch is, and what ``identical up to translation'' means.
Delaunay sets are an important example.
A Delaunay set of $\R^d$ is a subset $X \subset \R^d$, which is \emph{uniformly discrete} and \emph{relatively dense}.
These two conditions mean respectively that there is a constant~$m$ such that any two points in $X$ are separated by a distance at least~$m$; and a constant~$M$ such that any point in $\R^d$ is within distance at most $M$ of a point in~$X$.
The definition of \emph{patch} (or local configuration) of $X$ now becomes transparent when using the notation $X \cap B(x,r)$.
From there, the definitions of FLC, repetitivity and aperiodicity, as well as the definition of the hull (and its properties) are unchanged.

\subsection{Local derivations, and transversals}

 Given a tiling $T$, the associated tiling space $\Omega_T$ is sometimes called the ``continuous'' hull.
 It is also possible to describe a ``discrete'' tiling space, which is the analogue of a subshift in the symbolic setting (see Section~\ref{sec:symbolic}).
 Amongst other properties, it is a subset of $\Omega$, it intersects all $\R^d$-orbit and its intersection with any orbit is countable.
 It does not carry a $\Z^d$-action in general.
 This discrete tiling space is sometimes called a \emph{canonical transversal}.
 The canonicity of this transversal, however,  leaves a lot of room for choice.
 
 First, define what a local derivation is.
 \begin{defin}[see~\cite{BJS91}]\label{def:loc-deriv}
  Let $\Omega$ and $\Omega'$ be tiling spaces of repetitive and FLC tilings. A \emph{local derivation} is a factor map (\ie it is onto and commutes with translations) $\phi: \Omega \ra \Omega'$ which satisfies the following local condition:
  \[
   \exists r,R > 0, \ \forall T \in \Omega, \ \phi(T) \cap B(0,r) \text{ only depends on } T \cap B(0,R).
  \]
 \end{defin}
In particular, for all $x$, $\phi(T) \cap B(x,r)$ only depends on  $T \cap B(x,R)$.

 These conditions really mean that $\phi$ is a map defined on patches of size $R$, and $\phi(T)$ is obtained by gluing together the images under $\phi$ of the patches out of which $T$ is made. In this case, $\phi(T)$ is said to be locally derived from $T$.
 When $\phi$ is invertible and $\phi^{-1}$ is also a local derivation, $T$ and $\phi(T)$ are called \emph{mutually locally derivable}, or MLD.
 Such maps are the analogues of \emph{sliding block codes} in symbolic dynamics.

 \begin{defin}\label{def:xi}
  Let $\Omega$ be a tiling space, and $\ron D$ be a local derivation rule defined on $\Omega$, such that for all $T \in \Omega$, $\ron D (T)$ is a Delaunay set. Such a rule is called a \emph{local pointing rule}.
  The \emph{canonical transversal} associated with $\ron D$ is the set
  \[
   \Xi_\ron{D} = \{ T \in \Omega \tq 0 \in \ron D (T) \}.
  \]
 \end{defin}
 
 The Delaunay set $\ron D (T)$ is repetitive and has finite local complexity, since $T$ has these properties itself. It does not need to be MLD to $T$, or even aperiodic (see Figure~\ref{fig:point-domino}). 

 \begin{figure}[htp]
  \begin{center}
   \includegraphics[scale=0.75]{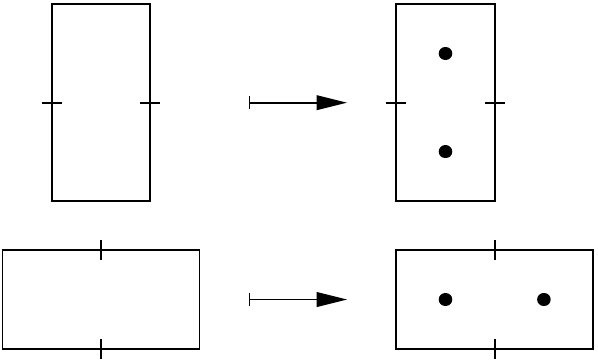}
  \end{center}
  \caption{A local pointing rule, as defined on a set of rectangular tiles of size $2\times 1$. Given any domino tiling made from these tiles (with dominoes meeting full-edge to full-edge), its associated point pattern will be periodic, even when the tiling is not.}
  \label{fig:point-domino}
 \end{figure}

 In the literature, ``the canonical transversal'' is usually built by choosing a pointing rule $\ron D$ which selects exactly one point in the interior of each prototile (for example their barycentre if the tiles are convex), and by pointing the tilings consistently.
 
 Canonical transversals in the sense above are abstract transversals in the sense of Muhly Renault and Williams~\cite{MRW87}. The reason why they are called ``canonical'' in spite of their apparent lack of canonicity is because $T$ and $T'$ are close in a canonical transversal if they agree \emph{exactly} on a large ball centred at the origin (not up to a small translation).
 
 \begin{rem}
  If $\Xi$ is a canonical transversal in a tiling space $\Omega$, the topology on $\Xi$ is induced by the family of clopen sets:
  \[
    U_P = \{T \in \Xi \tq P \subset T \},
  \]
  where $P$ is in the family of admissible patches. Equivalently, it is induced by the metric:
  \[
    D'(T,T') = \inf\Big( \big\{\eps > 0 \tq T \cap B(0,\eps^{-1}) = T' \cap B(0,\eps^{-1}) \big\} \cup \big\{ 1/\sqrt{2} \big\} \Big).
  \]
 \end{rem}

 Another description of the ``canonical'' transversals is that they are vertical with respect to a (truly) canonical solenoid structure on the tiling space.

 \begin{thm}[(see for ex.\ \cite{BBG06,BG03})]\label{thm:box}
  The tiling space associated with a repetitive, aperiodic and FLC tiling is an abstract (so-called \emph{flat}) solenoid in the following sense:
  for each $T \in \Omega$, there is a neighbourhood $V$ of $T$, and a chart map $\phi$, which maps $V$ homeomorphically to the direct product of a ball of $\R^d$ by a Cantor set.
  Furthermore, $\phi$ commutes with the action of $\R^d$ by translation whenever this action is defined.
  Also, the transition maps satisfy the property:
  \begin{equation}\label{eq:solenoid}
   \psi \circ \phi^{-1} (x, \xi) = ( x - t_{\psi \circ \phi^{-1}}, \xi'),
  \end{equation}
  where $\xi'$ depends continuously on $\xi$, but the translation vector $t_{\psi \circ \phi^{-1}}$ only depends on the chart maps.
 \end{thm} 

 \begin{proof}
  This is a partial proof in order to outline how the canonical solenoid structure is chosen. Property~\eqref{eq:solenoid} will not be proved here.

  Let $T \in \Omega$. Let $P = T \cap B(0,1)$ be a small patch of $T$ around the origin, and let $\eps > 0$ be small with respect with the inner radius of the tiles. Define the derivation rule
  \[
   \ron D (T') = \{x \in \R^d \tq P \subset T'-x\}.
  \]
  It is clearly a local rule, and the associated transversal $\Xi_{\ron D}$ is the set of all tilings which coincide with $T$ up to radius at least~$1$.
  If $T$ is repetitive, aperiodic, and has finite local complexity, these are well-known facts that $\Xi_{\ron D}$ is a Cantor set, and that $\ron D (T')$ is a Delaunay set for all $T'$.
  
  Now, consider the map
  \begin{displaymath}
   \begin{array}{lrcl}
    \phi: & \overline{B(0,\eps)} \times \Xi_{\ron D} & \lra & \Omega \\
          & (x,T') & \longmapsto & T'-x.
   \end{array}
  \end{displaymath}
  It is clearly continuous, by continuity of the action. It is one-to-one, because $\eps$ was chosen small. It is therefore a homeomorphism onto its image, by compactness of the source space. Finally, it is an exercise to check that the image of $\phi$ restricted to $B(0,\eps) \times \Xi_{\ron D}$ is open in $\Omega$. It provides the desired chart.
  It is important to remark that, by construction, the chart maps commute with the action of $\R^d$ whenever it is defined.
 \end{proof}

 It can be proved that the canonical transversals in the sense of Definition~\ref{def:xi} are exactly the vertical transversals for this solenoid structure, that is: subsets $\Xi \subset \Omega$ which can be covered by chart boxes in such a way that for each such box $V \simeq B(0,\eps_V) \times X_V$,
 \[
  \Xi \cap V \simeq \{0\} \times X.
 \]

\subsection{Complexity}

 The complexity function of a tiling $T$ is a function $p$, such that $p(r)$ counts the number of patches of size $r$ in $T$.
 This requires to define what ``patch of size $r$'' means in this context.
 While, in the $1$-dimensional symbolic setting, there is no question about what a word of length~$n$ is, the situation is not so obvious in the continuous case.
 However, it appears that the asymptotic properties of the complexity function are unchanged by picking any reasonable definition of ``patch of size $r$''.
 
 First, remark that a tiling $T$ needs to have FLC in order to have a well defined complexity function (at least, in terms of naively counting patches---see~\cite{FS12} for ways of defining complexity for tilings without FLC).
 Second, if a tiling $T$ is repetitive, then all tilings in its hull have the same complexity function. Therefore, for minimal tiling spaces, a complexity function is associated with the space rather than with one specific tiling.
 This justifies the following definition.
 
 \begin{defin}
  Let $\Omega$ be a space of tilings which are repetitive, FLC and aperiodic.
  Let $\ron D$ be a local pointing rule, and $\Xi$ the associated transversal.
  Then, the complexity function associated with these data, noted $p_\Xi$ or $p_{\ron D}$ is:
  \begin{align*}
    p_\Xi (r) & = \card \big\{ T \cap B(0,r) \tq T \in \Xi \big\} \\
      & = \card \big\{ (T_0 - x) \cap B(0,r) \tq x \in \ron D (T_0) \big\} \text{ for any } T_0 \in \Omega.
  \end{align*}
 \end{defin} 

 According to this definition, the complexity function counts the number of \emph{pointed patches} of size $n$, where the pointing has to be chosen in advance and is not canonical to the tiling. However, it turns out that changing the transversal doesn't change the complexity function.
 
 \begin{prop}\label{prop:complex-indep-xi}
  Let $\Omega$ be a tiling space, and $\Xi_1$, $\Xi_2$ be two canonical transversals associated with pointing rules $\ron D_1$ and $\ron D_2$ respectively.
  Then the complexity functions $p_{\Xi_1}$ and $p_{\Xi_2}$ are equivalent in the following sense:
  \[
   \lambda_1 p_{\Xi_2} (r - c) \leq p_{\Xi_1} (r)  \leq  \lambda_2 p_{\Xi_2} (r + C)
  \]
  for some constants $c,C$, $\lambda_1, \lambda_2 > 0$ and all $r$ big enough.
 \end{prop}

 \begin{proof}
  Let $T \in \Omega$. First, let ${\ron D}$ be the pointing rule defined by
  \[
   {\ron D} (T) = \ron D_1 (T) \cup \ron D_2 (T).
  \]
  It is a pointing rule itself: ${\ron D} (T)$ is locally derived from $T$, it is relatively dense, and the uniform discreteness comes from finite local complexity of $T$ (and hence, finite local complexity of ${\ron D} (T)$).
  
  Let $\Xi$ be the transversal defined by $\ron D$. It is sufficient for proving the theorem to show that $p_\Xi$ is equivalent (in the sense above) to $p_{\Xi_i}$, for $i = 1,2$. Let us prove it for $i=1$.
  
  For any given $r$, let $X_r$ be the set of all patches of the form $(T-x) \cap B(0,r)$, with $x \in \ron D (T)$.
  Let $Y_r$ be the similarly defined set with $x \in \ron D_1 (T)$.
  It is obvious that $p_{\Xi_1}(r) \leq p_{\Xi}(r)$ for all~$r$ (since $Y_r \subset X_r$).
  
  Conversely, let $R$ be the constant such that for all $x \in \R^d$, $B(x,R)$ intersects $\ron D_1 (T)$.
  Let $\lambda > 0$ be the maximum number of points in $\ron D (T) \cap B(x,R)$, for $x \in \R^d$. It is well defined and finite (because of volume considerations using uniform discreteness of $\ron D (T)$, for example).
  For all patch $P \in X_r$, there is a patch $P' \in Y_{r+R}$ such that $P$ is a subpatch of $P'$; this patch $P'$ can contain at most $\lambda$ distinct translates of patches in $X_r$.
  This defines a map $X_r \ra Y_{r+R}$; each element in $Y_{r+R}$ has at most $\lambda$ preimages under this map.
  Therefore,
  \[
   p_\Xi (r) \leq \lambda p_{\Xi_1} (r+C).
  \]

 \begin{figure}[ht]
 \begin{center}
  \includegraphics[scale=0.8]{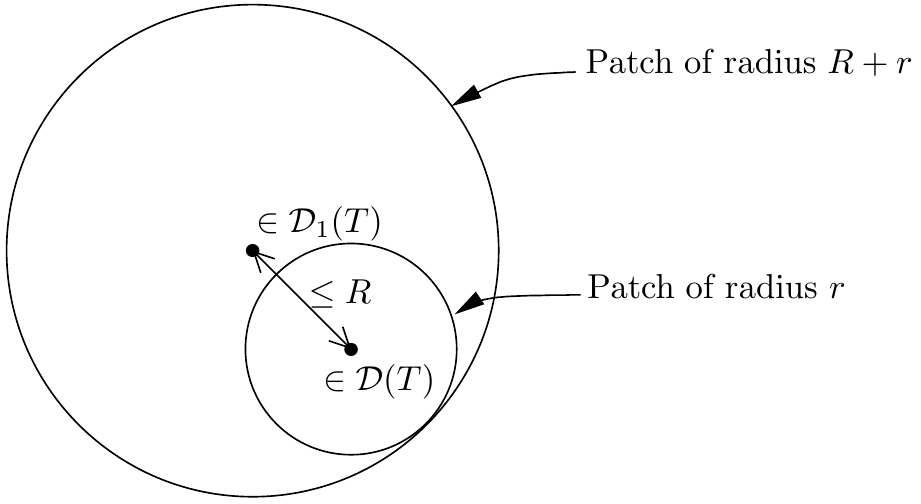}
 \end{center}
 \caption{$\ron D_1$ is sparser than $\ron D$; yet, any patch of size $r$ centered at a point of $\ron D$ is a subpatch of some patch of size $R+r$ centred at a point of $\ron D_1$.}
 \label{fig:indep-xi}
\end{figure}
 \end{proof} 

 Since the choice of a transversal does not change the asymptotic behaviour of the complexity function, it is acceptable to write ``the complexity function of $T$'' when there is no risk of ambiguity. It is understood that an unspecified transversal is chosen, the choice of which is not important.

 \begin{rem}
  It can also be proved that these pointed complexity functions $p_\Xi$ are equivalent to an unpointed complexity function $p$, where $p(r)$ is the number of equivalence classes up to translation of patches of the form $T \cap B(0,r)$, for $T \in \Omega$. It is proved in~\cite[Proposition~1.1.20]{Jul-phd}.
 \end{rem}

 Changing the transversal does not change the complexity. It is worth noticing that applying an invertible local derivation to a tiling also preserves the complexity function.
It can be proved directly by remarking that a local derivation is essentially a map from patches to patches.
The proof is not included here. See Theorem~\ref{thm:frank-sadun} for a more general result.

 \begin{prop}
  Let $T \in \Omega$ and $T' \in \Omega'$ be two tilings which are mutually locally derivable. Let $\ron D (T)$ be a Delaunay set derived from $T$ (and hence locally derived from $T'$ by mutual local derivability of $T$ and $T'$). Call $\Xi \subset \Omega$ and $\Xi' \subset \Omega'$ the canonical transversals relative to these Delaunay sets. Then their complexity functions satisfy:
  \[
   p_\Xi(r-c) \leq p_{\Xi'}(r) \leq p_\Xi (r+C),
  \]
  for $r$ big enough.
 \end{prop}

 The equivalences between complexity functions in the propositions above are quite fine. It is appropriate for our purpose to have a coarser notion of equivalence.
 \begin{defin}\label{def:cplex-equiv}
  Two complexity functions $p_1$ and $p_2$ are \emph{equivalent} if there are $C_1,C_2,m,M > 0$ such that for all $r$ big enough:
  \[
   C_1 p_2(mr) \leq p_1(r) \leq C_2 p_2(Mr).
  \]
 \end{defin}
 
 Remark that it is indeed a coarser notion that the ones above. However, the notion of a ``patch of size $r$'' is itself not uniquely defined: it depends on a choice of norm of $\R^d$.
 By equivalence of norms in finite dimension, two different choices of norms would produce equivalent complexity functions in the sense above.
 Therefore, and unless there is a reason to prefer one norm over others, it is fine to use the coarser version.

\subsection{Repetitivity}

 The repetitivity of a tiling measures how often patches of a given size repeat.
 It was introduced in the setting of symbolic dynamics by Hedlund and Morse~\cite{MH38}.
 In a more general setting, it has been studied by Lagarias and Pleasants for Delaunay sets~\cite{LP03}.
 
 \begin{defin}
  Let $\Omega$ be a repetitive, aperiodic and FLC tiling space, and $\ron D$ be a pointing rule defining a transversal $\Xi$.
  The repetitivity function is a function $r \mapsto \ron R_{\ron D}(r)$ or $\ron R_\Xi(r)$ defined by
  \[
   \ron R_\Xi (r) = \sup \{ R > 0 \ ; \ \forall T,T_0 \in \Xi, \ \exists x \in B(0,R), \ T-x \cap B(0,r) = T_0 \cap B(0,r)  \}.
  \]
  In words: $\ron R_\Xi(r) < C$ means that for any pointed patch of size $r$ of the form $T_0 \cap B(0,r)$ and for any $T \in \Xi$, said patch occurs in $T$ within distance at most $C$ of the origin (i.e.\ $T-x$ and $T_0$ agree up to distance $r$ with $\nr{x} \leq C$).
 \end{defin}

 The repetitivity function can be interpreted in dynamical terms. Let $\Omega$ be a tiling space with canonical transversal $\Xi$.
 Then for each $r$ there is a partition $\ron P_r$ of $\Xi$ in clopen sets of the form 
 \[
   \Xi_{T,r} := \{T' \in \Xi \ ; \ T \cap B(0,r) = T' \cap B(0,r) \}.
 \]
 Then $\ron R_\Xi(r)$ measures the infimum of the length of the return vectors from any point in $\Xi$ to any element of the partition $\ron P_r$.

 Once again, it can be proved that the repetitivity function of a tiling space is independent of the choice of the transversal.
 
 \begin{prop}
  Let $\Omega$ be an aperiodic repetitive FLC tiling space, and $\Xi_1$, $\Xi_2$ be two canonical transversals corresponding to pointing rules $\ron D_1$ and $\ron D_2$ respectively.
  Then
  \[
   \ron R_{\Xi_2} (r - c) - K \leq \ron R_{\Xi_1} (r)  \leq  \ron R_{\Xi_2} (r + c) + K
  \]
  for some constants $c,K \geq 0$ and all $r$ big enough.
 \end{prop}
 
 Lagarias and Pleasants~\cite{LP03} showed that for aperiodic repetitive tilings, the repetitivity function grows at least linearly.
 It is therefore harmless to drop the constant $K$ in the theorem above, up to adjusting~$c$.

 \begin{proof}
  As in the proof of the analogous result for complexity (Proposition~\ref{prop:complex-indep-xi}), we may assume without loss of generality that $\Xi_1 \subset \Xi_2$.
  
  It is straightforward that $\ron R_{\Xi_1}(r) \leq \ron R_{\Xi_2}(r)$.
  Conversely, let $R_0$ be the a radius (obtained by relative density) such that for any $T \in \Xi_2$, there exists $x \in B(0,R_0)$ such that $T-x \in \Xi_1$.
  Let now be $r> 0$ and $T,T_0 \in \Xi_2$ be given.
  Let us show that there is a $y \in \R^d$ with $\nr{y} \leq 2R_0 + \ron R_{\Xi_2}(r+R) + \eps$ for arbitrarily small $\eps$ such that $(T-y) \cap B(0,r) = (T_0-y) \cap B(0,r)$.
  Let $x$ and $x_0$ be of norm at most $R_0$ such that $T - x$ and $T_0 -x_0$ are in $\Xi_1$.
  Next let $y_0$ be a vector such that $(T-x) - y_0$ and $T_0-x_0$ agree up to radius $r+R$.
  By definition of the repetitivity function, it can be chosen of norm at most $\ron R_{\Xi_1}(r+R)+\eps$ for any $\eps > 0$.
  Now, it means that $(T-x-y_0) +x_0$ and $T_0$ agree up to radius $r$, since $\nr{x_0} \leq R_0$.
  So if $y=y_0+x-x_0$, then $T-y$ and $T_0$ agree up to radius $r$.
  We have control on the norm of $y$, so that:
  \[
   \ron R_{\Xi_2}(r) \leq \ron R_{\Xi_1}(r+R_0) + 2 R_0.
  \]
 \end{proof}

\subsection{Formalism of symbolic dynamics}\label{sec:symbolic}

 Tiling spaces also appear as suspension of subshifts. A $d$-dimensional subshift is a subset of $\ron A^{\Z^d}$, where $\ron A$ is a finite set of symbols (finite alphabet).
 There is a natural action $\sigma$ of $\Z^d$ on this set, defined by
 \[
  [\sigma^n (\omega)]_k = \omega_{k-n},
 \]
 for $n \in \Z^d$.
 
 Given $\omega \in \ron A^{\Z^d}$, the subshift associated with $\omega$ is the smaller subset of $\ron A^{\Z^d}$ which contains $\omega$, is closed and is shift-invariant. Equivalently, it is the closure of the orbit of $\omega$ for the product topology.
 
 The notions of aperiodicity and repetitivity are straightforward. Finite local complexity is automatic if $\ron A$ is finite.
 
 \begin{defin}
  Let $\Xi$ be an aperiodic, repetitive subshift on $\Z^d$. The \emph{suspension} of $\Xi$ is the following space:
  \[
   \ron S \Xi = \Xi \times \R^d / \sim
  \]
  where $\sim$ is the equivalence relation generated by
  \[
   (\omega , x) \sim (\sigma^n (\omega), x-n), \quad n \in \Z^d.
  \]
  It carries naturally a $\R^d$-action by translations.
 \end{defin} 

 In practice, it is convenient to view elements of $\Xi$ as tilings by $d$-cubes of edge-length~$1$, and $(\omega,x) \in \ron S \Xi$ can be interpreted as $\omega+x$.
 The tiling topology, as defined above, coincides with the quotient topology on this suspension.

 The following definition is standard.
 \begin{defin}
  Two subshifts $\Xi_1$ and $\Xi_2$ are said to be \emph{flow-equivalent} if their suspensions are homeomorphic.
 \end{defin}

 The main results of this paper can therefore be interpreted in a symbolic setting: they measure how the complexity and repetitivity function of two subshifts relate when the subshifts are flow equivalence.

\section{Homeomorphisms between tiling spaces}

 \subsection{Statement of the results}\label{ssec:statement}

  In this paper, the basic setting is to consider a homeomorphism between two aperiodic repetitive tiling spaces with finite local complexity.
  What can be said about the underlying tilings if the two tiling spaces are homeomorphic?
  
  The first result is concerned with the complexity function of the tiling spaces: if two spaces are homeomorphic, the complexity functions are asymptotically equivalent in the sense of definition~\ref{def:cplex-equiv}.
  If the complexity is high, then the tiling systems may have entropy.
  In this case, the entropy is not preserved by homeomorphism (even though the fact that it is positive is).
  However, if the complexity grows polynomially, then the exponent is preserved by homeomorphism.

 \begin{thm}\label{thm:main}
  Let $h: \Omega \lra \Omega'$ be a homeomorphism between two aperiodic repetitive tiling spaces with finite local complexity.
  Then there exist a transversal $\Xi \subset \Omega$, a transversal $\Xi' \subset \Omega'$ and constants $c,C,m,M > 0$ such that for all $r$ big enough,
  \[
   c p_{\Xi'} (m r) \leq p_\Xi (r) \leq C p_{\Xi'} (M r).
  \]
 \end{thm}
 
 \begin{cor}
  Let $\Omega$ be a (repetitive, aperiodic, FLC) tiling space, with polynomial complexity function, which means that there exists an $\alpha$ such that for some $\Xi$ (and hence for all $\Xi$), there are two constants $C_1$ and $C_2$ such that for all $r$ big enough,
  \[
   C_1 r^\alpha \leq p(r) \leq C_2 r^\alpha.
  \]
  Then the exponent $\alpha$ is a homeomorphism invariant of the space: any tiling space which is homeomorphic to $\Omega$ has polynomial complexity function with the same exponent $\alpha$.
 \end{cor}
 
 In the language of symbolic dynamics, the results above can be stated as follows: If two minimal aperiodic $\Z^d$-subshifts are flow-equivalent, then their complexity functions are equivalent in the sense above.

 It is interesting to combine these results with some other results obtained previously on complexity.
 It was remarked in~\cite{JS11} that the complexity of a tiling can be related with the box-counting dimension of the transversals for the metric defined in Section~\ref{ssec:tilings}.
\begin{prop}[\cite{JS11}]
 Given a tiling space $\Omega$ and a canonical transversal $\Xi$, endowed with the metric defined before, the box-counting dimension of $\Xi$ is given by the formula:
 \[
 \dim_{\mathrm B} (\Xi) = \lim_{r \ra +\infty} \frac{\log(p_\Xi(r))}{\log(r)}
 \]
 when it exists.
\end{prop}

 \begin{cor}\label{cor:dimbox}
  The box-counting dimension of the transversals of FLC, aperiodic, repetitive tiling spaces is preserved by homeomorphism between the tiling spaces.
 \end{cor}
 It is striking to see that a homeomorphism between two tiling spaces preserves a quantity which is essentially \emph{metric}.
 
 It is not the first time that links between topology and complexity are investigated.
 In~\cite{Jul10}, some relationships were established between the rank of the cohomology groups over the rationals on the one hand, and the asymptotic growth of the complexity function.
 For example in dimension~$1$, a tiling space with at most linear complexity has finitely generated \v{C}ech cohomology groups over the rationals.
 For cut-and-project tilings, the interplay is even richer: the complexity function is bounded above by $r^d$ (with $d$ the dimension of the space) if and only if the rational cohomology groups are finitely generated.
 The results of this paper push further in this direction: the complexity is not only loosely related with the topology of the space; it really is homeomorphism-invariant.

 The repetitivity function of a tiling is another quantity which is preserved by homeomorphism, as shown by the following result.
 
 \begin{thm}\label{thm:repet}
  Let $h: \Omega \lra \Omega'$ be a homeomorphism between two aperiodic repetitive tiling spaces with finite local complexity.
  Then there exist a transversal $\Xi \subset \Omega$ and a transversal $\Xi' \subset \Omega'$ and four constants $c,C,m,M > 0$ such that for all $r$ big enough,
  \[
   c \ron R_{\Xi'} (m r) \leq \ron R_\Xi (r) \leq C \ron R_{\Xi'} (M r).
  \] 
 \end{thm}

 The study of the repetitivity function is a quantitative approach. It is somehow more common in the literature to make the qualitative distinction between \emph{linearly repetitive} tiling spaces and the others.
 A tiling space is linearly repetitive if its repetitivity function grows no faster than $r \mapsto \lambda r$ for some constant $\lambda$.
 Linearly repetitive tilings are amongst those with the highest degree of regularity. Many of the tilings used to model quasicrystals have this property, including the well-known Penrose tilings.
 This result then follows from the theorem.
 
 \begin{cor}
  Let $h: \Omega \lra \Omega'$ be a homeomorphism as in the previous theorem.
  Then $(\Omega,\R^d)$ is linearly repetitive if and only if $(\Omega',\R^d)$ is.
 \end{cor}

 \subsection{Homeomorphisms and tiling groupoids}
 
 Given a dynamical system, it is possible to define a groupoid which carries information on the space, the group, and the action of the latter on the former. Groupoids also serve as a generalisation of group action, when there is no group.
 For example, a tiling space carries an action of $\R^d$.
 In one dimension, this action of $\R$ restricts to an action of $\Z$ on a transversal~$\Xi$ (by the first return map).
 In higher dimension, there is in general no group action on the transversal; however, a groupoid can be defined as a replacement.
 
 While the groupoid point-of-view can be enlightening, it is not essential for most of the paper. Groupoids can be ignored until the Section~\ref{sec:deformations}.
 The crucial elements of this section---which shouldn't be skipped at first---are Definition~\ref{def:hT} and Theorem~\ref{thm:hT} below.
 
 \begin{defin}
  A groupoid $G$ is defined as a small category with inverses. It consists of a set of base points $G^{(0)}$, and a set of arrows $G$ between them.
  Given $x \in G$, it has a source and a range $s(x)$ and $r(x)$ in $G^{(0)}$ and the product $xy$ is well defined if $r(y) = s(x)$.
  For all $p \in G^{(0)}$ there is an element $e_p \in G$ which is neutral for multiplication (and $p$ is identified with $e_p$, so that $G^{(0)} \subset G$), and for all $x$ there is $x^{-1}$ such that $xx^{-1} = r(x)$, $x^{-1}x = s(x)$.
  Finally, the product is associative, whenever it is defined.
 \end{defin}
 
 \begin{defin}
  A tiling space $\Omega$ has an associated groupoid, noted $\Omega \rtimes \R^d$, which is defined as the direct product of $\Omega$ and $\R^d$, with range map, source map and product defined as follows:
  \begin{gather*}
   s(T,x) = T,   \qquad \qquad r(T,x) = T-x, \\
   (T,x) \cdot (T',x') = (T',x+x') \quad \text{ if } \quad T'-x'=T.
  \end{gather*}
  The topology on this groupoid is that induced by $\Omega \times \R^d$.
 \end{defin}

 Consider a homeomorphism $h: \Omega \lra \Omega'$ between two aperiodic, FLC and minimal tiling spaces.
 By Theorem~\ref{thm:box}, the path-connected components in such tiling spaces are the orbits. Since a homeomorphism has to send path-connected component to path-connected component, we can define the following map.
 
 \begin{defin}\label{def:hT}
  Let $h: \Omega \lra \Omega'$ be a homeomorphism. Define for all $T \in \Omega$ a function $h_T : \R^d \ra \R^d$ by
  \[
   h(T - x) = h(T) - h_T(x).
  \]
  It is well-defined, by aperiodicity of the tilings with which we are dealing.
 \end{defin}
 
 \begin{thm}\label{thm:hT}
  Given a homeomorphism $h : \Omega \lra \Omega'$ between two aperiodic, minimal, FLC tiling spaces, the function
  \begin{displaymath}
    \begin{array}{ccl}
     \Omega \times \R^d & \lra & \R^d \\
     (T,x) & \longmapsto & h_T(x)
    \end{array}
  \end{displaymath}
  is continuous in two variables. Furthermore, for any fixed $T$, $h_T$ is a homeomorphism, and
  \[
   \big( h_T \big)^{-1} = \big(h^{-1}\big)_{h(T)}.
  \]
 \end{thm}
 
 \begin{proof}
  The first step of this proof is to show that for any given $T \in \Omega$, the map $x \mapsto h_T(x)$ is continuous.
  The map $h_T$ is defined as the bijection which makes the following diagram commute:
  \[
   \xymatrix{\ar @{} [dr] |{\circlearrowleft}
     \R^d \ar[d] \ar[r]^{h_T}   & \R^d  \ar[d] \\
     \orb(T) \ar[r]^-{h}        & \orb(h(T))
   }
  \]
  where $\orb(T)$ is the orbit of $T$ in $\Omega$ under the action of $\R^d$.
  The difficult point is that the action maps $\R^d$ continuously and bijectively onto $\orb(T)$ via $x \mapsto T-x$; however, for the relative topology induced by that of $\Omega$, $\orb(T)$ is not homeomorphic to $\R^d$, and not even locally homeomorphic. Indeed, the image of an open set is not open (in general) for the induced topology, and $\orb(T)$ is not locally connected.
  However, Theorem~\ref{thm:box} gives information about the local structure of the leaves, which allows to prove that $h_T$ is a homeomorphism.
  
  Here, the sequential characterization of continuity is used: $T$ and $x$ being fixed, let $(x_n)_{n \in \N}$ be a sequence converging to $x$.
  Let $\gamma: [0,1] \ra \R^d$ be the continuous path from $x_1$ to $x$, such that ${\gamma}$ restricted to ${[1-2^{-n}, 1-2^{-(n+1)}]}$ is the constant speed parametrization of the segment $[x_{n+1},x_{n+2}]$.
  To prove continuity of $h_T$ at $x$, it is enough to prove that $h_T \circ \gamma$ tends to $h_T (x)$ as $t$ tends to~$1$.
 
  Since $h$, $\gamma$ and of the action of $\R^d$ on $\Omega$ are continuous, $h(T - \gamma(t))$ is close to $h(T-x)$ provided $t$ is close to~$1$. Assume $t_0$ is close enough to $1$ so that for all $t > t_0$, $h(T - \gamma(t))$ is in a chart box of the form $V \simeq_\phi B(0,\eps) \times X$ around $h(T-x)$, and such that $\phi(h(T-x)) = (0,\xi)$.
  
  For $t > t_0$, the path $t \mapsto \phi \circ h (T - \gamma (t))$ is a continuous path in $B(0,\eps) \times X$, and therefore must be included in a path-connected component of the form $B(0,\eps) \times \{\xi'\}$, and therefore be of the form $(\eta(t), \xi')$. Since $\phi \circ h (T-\gamma(t))$ needs to tend to $(\xi,0)$, then $\xi'=\xi$.
  
  Finally,
  \begin{align*}
   \phi \circ h (T - \gamma (t)) & = \phi ( h(T) - h_T \circ \gamma (t)) \subset B(0,\eps) \times \{\xi\}, \\
               & = (\eta(t), \xi).
  \end{align*}
  Now, $\eta(t)$ tends to zero on the one hand, and is equal to $h_T(\gamma(t)) - h_T(x)$ on the other hand (because $\phi$ commutes with the action of $\R^d$ whenever it is defined, see Theorem~\ref{thm:box}). Therefore, $h_T(\gamma(t))$ tends to $h_T(x)$, and $h_T$ is continuous.

  The second part of the proof is to show that the function is continuous of two variables.
  Let $(T_0,x_0) \in \Omega \times \R^d$.
  What needs to be shown is
  \begin{multline*}
    \forall \eps > 0 \ ;\  \exists \delta > 0 \ ; \ \forall  (T,x) \in \Omega \times \R^d, \\
    \Big( \nr{x-x_0} < \delta \text{ and } D(T,T_0) < \delta \Big)  \Rightarrow \nr{h_T(x) - h_{T_0}(x_0)} < \eps.
  \end{multline*}

  Let $\eps > 0$.
  Let $\delta_0$ be such that whenever $T_1,T_2$ are $\delta_0$-close, then their images by $h$ are $\eps / 3$-close (using uniform continuity of $h$).
  Finally, let
  \[
  \delta = \min\big\{ (\delta_0^{-1} + \nr{x_0} + 2\delta_0)^{-1}, \delta_0/2 \big\}.
  \]
  Let now $T \in \Omega$ in a $\delta$-neighbourhood of  $T_0$. Then, up two small translation of combined size less than $2 \delta < \delta_0$, the tilings $T$ and $T_0$ agree on $B(0,\delta_0^{-1} + \nr{x_0} + \delta_0)$.
  Then, for all $y_0 \in \overline{B(0,\nr{x})}$ and all $y \in \R^d$ with $\nr{y-y_0} < \delta \leq \delta_0/2$, the tilings $T-y$ and $T-y_0$ agree at least on $B(0,\delta_0)$ up to a small (less than $\delta$) translation.
  This last statement holds in particular for $y= \lambda x$ ($\lambda \in [0,1]$).
 
  Let $x \in B(x_0,\delta)$, and $T \in \Omega$ be $\delta$-close from $T_0$.
  Then, for all $\lambda \in [0,1]$, $h(T - \lambda x)$ and $h(T_0 - \lambda x_0)$ are $\eps / 3$-close.
  Assume (by contradiction), that $\nr{h_T (x) - h_{T_0} (x_0)} \geq \eps$.
  Since $h_T$ and $h_{T_0}$ are continuous, and $\nr{h_T(0) - h_{T_0} (0)} = 0$, one can use the intermediate value theorem: there exists a $\lambda_0 \in [0,1]$ such that $\nr{h_T(\lambda_0 x) - h_{T_0} (\lambda_0 x_0)} = \eps$.
  One has
  \begin{align*}
   h(T - \lambda_0 x) & = h(T) - h_T(\lambda_0 x)  \\
   h(T_0 - \lambda_0 x_0) & = h(T_0) - h_{T_0} (\lambda_0 x_0).
  \end{align*}
  On the one hand, $h(T)$ and $h(T_0)$ match around the origin up to a translation smaller than~$\eps/3$.
  On the other hand, so do the left-hand sides of the equations above.
  Besides, $h_T (\lambda_0 x)$ and $h_{T_0} (\lambda_0 x_0)$ differ by exactly $\eps$, which $\eps$ is much smaller than the radius of a tile. It is a contradiction.
 \end{proof}
 
 \begin{rem}
  The map $(T,x) \mapsto h_T(x)$ is actually a groupoid morphism (where the additive group $\R^d$ is seen as a groupoid with one basepoint).
  It results from a direct check:
  \[
   h_T(x+y) = h_T(x) + h_{T-x}(y).
  \]
  This relation can also be called a \emph{cocycle condition} for a reason which will be addressed in the last section of this paper.
 \end{rem}

 \begin{rem}
  As a consequence of this result, the following map
  \[
   \begin{array}{ccl}
      \Omega \rtimes \R^d & \lra & \Omega' \rtimes \R^d \\
       (T,x) & \longmapsto & (h(T),h_T(x))
   \end{array}
  \]
  is an isomorphism of topological groupoids. Therefore, by a theorem of Renault~\cite{Ren80}, the associated cross product $C^*$-algebras $C(\Omega) \rtimes \R^d$ and $C(\Omega') \rtimes \R^d$ are isomorphic.
  It appears that much of the dynamics is constrained by the topology of the tiling spaces, because of the local product structure.
  If the topological spaces are homeomorphic, the homeomorphism is automatically an orbit equivalence, and the groupoids and associated $C^*$-algebras---which encode the dynamics---are isomorphic.
 \end{rem}

 \subsection{Strategy for the proof}
 
 Frank and Sadun~\cite{FS12} have proved that whenever two tiling spaces are topologically conjugate, their complexity functions are equivalent.
 Their result involves a different complexity function as the one which is used here, as it can measure complexity of tiling spaces without FLC.
 Here is how their result would read in the framework of the present paper.
 
 \begin{thm}\label{thm:frank-sadun}
  Let $h: \Omega \lra \Omega'$ be a topological conjugacy between two tiling spaces (\ie a homeomorphism which commutes with the actions of $\R^d$ by translation).
  We assume that the tiling spaces are aperiodic, minimal, with FLC.
  Then, there are two canonical transversals $\Xi \subset \Omega$ and $\Xi' \subset \Omega'$, such that
  the associated complexity function $p$ and $p'$ satisfy
  \[
   m p(r-a) \leq p'(r) \leq M p(r+b).
  \]
  for some constants $a$, $b$, $m$, $M$, and all $r$ big enough.
 \end{thm}
 
 Note that by Proposition~\ref{prop:complex-indep-xi}, it is not important which transversals $\Xi$ and $\Xi'$ are chosen, as any complexity functions would then satisfy this inequality.
 
 We include a sketch of the proof, with a simplifying assumption that $h$ sends a vertical transversal to a vertical transversal.
 Even in this simplified setting, the proof contains important ideas which will be used in the proof of Theorem~\ref{thm:main}.
 It is pointed out afterwards how these ideas are to be generalised.

 \begin{proof}
  Let $\Xi \subset \Omega$ be any canonical transversal. It is assumed that $\Xi' := h(\Xi)$ is a canonical transversal itself (this is a simplifying assumption, as it shouldn't be expected in general).
 
  The goal is then to compare the functions $p$ and $p'$ associated with $\Xi$ and $\Xi'$ respectively.
  Let $\eps < 2^{-1/2}$. By uniform continuity of $h$, there is a $\delta < 2^{-1/2}$ such that for all $T,T' \in \Xi$, if $T$ and $T'$ are at least $\delta$-close, then $h(T)$ and $h(T')$ are at least $\eps$-close.
  In terms of patches, it means that if $T$ and $T'$ agree on $B(0,\delta^{-1})$, then their images agree on $B(0,\eps^{-1})$. In particular, $h(T)$ and $h(T')$ have the same local configuration at the origin.
  
  Now, $\eps$ and $\delta$ being fixed, let $r > 0$, big enough so that it is big compared to $\eps^{-1}$ and $\delta^{-1}$.
  Assume $T$ and $T'$ agree on a ball of radius $r+ \eps^{-1}$. Then it means that for all $x \in \R^d$ with $\nr{x} \leq r$, $T-x$ and $T'-x$ agree on $B(0,\delta^{-1})$.
  Using the fact that $h$ is a conjugacy, it means that for all such $x$, $h(T)-x$ and $h(T')-x$ agree on $B(0,\eps^{-1})$. Therefore, $T$ and $T'$ agree up to radius $r+\delta^{-1}$ implies that their images agree up to radius $r+\eps^{-1}$.
  Therefore,
  \[
    p(r+\delta^{-1}) \geq p'(r+\eps^{-1}).
  \]
  An exchange of $h$ and $h^{-1}$ provides the other inequality.
 \end{proof}
 
 This proof has three essential ingredients:
 \begin{enumerate}
  \item The image of a transversal is a transversal (simplifying assumption);
  \item The map $h$ is uniformly continuous, and it allows to control how two tilings agree on a small neighbourhood of the origin when their preimages do;
  \item The conjugacy sends orbits to orbits in a trivial way (given $T$, the map $h_T$ is the identity map), so that patches of tilings in $\Omega$ can be compared to patches of tilings in $\Omega'$.
 \end{enumerate}
 
 In the case when $h$ is just a homeomorphism, point (2) is still satisfied. Point (1) is not satisfied, however it is very close to being satisfied: for all $\eps > 0$, there are transversals $\Xi \subset \Omega$ and $\Xi' \subset \Omega'$ such that $h(\Xi)$ and $\Xi'$ are within distance~$\eps$.
 This important approximation lemma is proved in the next section.
 
 The main difficulty is point (3). Given a ball of radius $r$ in an orbit, what is the radius of its image by $h$? We need to control how $h_T (B(0,r))$ grows with~$r$.
 Lemma~\ref{lem:quasi-lipschitz} and its corollary show that $h_T$ is not far from being Lipschitz.

\section{An approximation lemma}

The following key lemma addresses one of the points outlined in the strategy of the proof: a homeomorphism between two tiling spaces is very close to sending a given canonical transversal to a canonical transversal.
It is an interesting result in its own right.
It should be compared with results obtained by Rand and Sadun~\cite{RS09}: in their paper, it is proved that any continuous map between FLC aperiodic repetitive tiling spaces can be approximated arbitrarily well by a \emph{local} map---a local map has the particularity that it sends \emph{any} canonical transversal to a canonical transversal.
However, the approximating map in their case is not guaranteed to be a homeomorphism when the original map is.
Note also that the map $g$ in the proof of Lemma~\ref{lem:quasi-lipschitz} is local in the sense of Rand--Sadun, but is \emph{a priori} not within a small distance of $h$, and is not guaranteed to be a homeomorphism either.

\begin{lem}\label{lem:redressement}
 Let $h: \Omega \lra \Omega'$ be a homeomorphism. Then for all $\eps > 0$, there is a vertical transversal $\Xi \subset \Omega$, a vertical transversal $\Xi' \subset \Omega'$, and a homeomorphism $h_\eps: \Omega \lra \Omega'$ such that
 \begin{enumerate}
   \item $h_\eps(\Xi) = \Xi'$;
   \item $h$ and $h_\eps$ are isotopic;
   \item there is a continuous function $s:\Omega \ra \R^d$, the norm of which is uniformly bounded by $\eps$ and such that $h_\eps = h-s$.
 \end{enumerate}
\end{lem}

\begin{proof}
 Let $\eps > 0$, which is assumed to be small compared to the inner radius of the tiles. By uniform continuity of $h$, there is a $\delta > 0$ such that $D(T,T') \leq \delta$ implies $D(h(T),h(T')) \leq \eps$.
 Then, there is a chart box $V_0 \subset \Omega'$, of diameter less than $\eps$. Let $U_0$ be a chart box in $\Omega$ such that $U_0 \subset h^{-1}(V_0)$. Then we have the following picture for some $\delta' \leq \delta$ and $\eps' \leq \eps$.
 \[
   \xymatrix{
   U_0 \ar[rr]^h \ar [d]_\simeq   && V_0 \ar[d]^\simeq \\
   B(0,\delta') \times X_0 \ar [rr]^\phi && B(0,\eps') \times X'_0
   }
 \]
 where $\phi$ is induced by $h$ (so that the diagram commutes). See also Figure~\ref{fig:vois-cstrct}.
 When no confusion can arise, $\phi$ will not be used, and this map will just be noted $h$, possibly with the precision ``in coordinates''.
 \begin{figure}[ht]
 \begin{center}
  \includegraphics{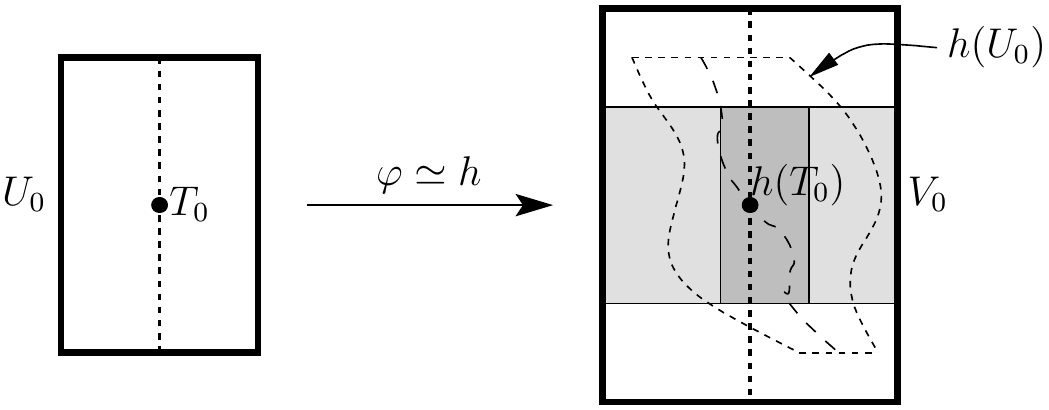}
 \end{center}
 \caption{$U_0$ and $V_0$, as initially built. They need pruning: for example $V_0$ is too tall.
    First, pick a box-shaped neighbourhood of $h(T_0)$ inside $h(U_0)$ (dark-grey neighbourhood~$V'$). Then, widen it (light-grey neighbourhood~$V$). Finally, trim $U_0$ so that its image is now included in the light-grey neighbourhood.}
 \label{fig:vois-cstrct}
\end{figure}

 With a good choice of charts, we can assume that $\phi (0,\xi) = (0,\xi')$ for some $\xi \in X_0$ and $\xi' \in X'_0$.
 Call $T_0$ and $T'_0$ the tilings corresponding to $(0,\xi)$ and $(0,\xi')$ respectively via the chart maps (so that $h(T_0) = T'_0$).
 Note that $h(U_0)$ is open in $V_0$. Therefore, there is a neighbourhood of $T'_0$, say $V' \simeq  B(0,\eps'_1) \times X'$ such that $V' \subset h(U_0)$.
 Let now $V \simeq B(0,\eps') \times X'$, and $U$ be the largest open subset of $U_0$ such that $h (U) \subset V$. It is easy to check that it is of the form $B(0,\delta') \times X$, with $X$ a clopen subset of~$X_0$.
 
 Let $\Xi$ and $\Xi'$ be the vertical transversals included in $U$ and $V$ respectively, such that $\Xi \simeq \{0\} \times X$ and $\Xi' \simeq \{0\} \times X'$.

We now have $U$ and $V$ two chart domains with $h(U) \subset V$, and two transversals $\Xi \subset U$ and $\Xi' \subset V$ such that $\Xi' \subset h(U)$.
\begin{figure}[ht]
 \begin{center}
  \includegraphics{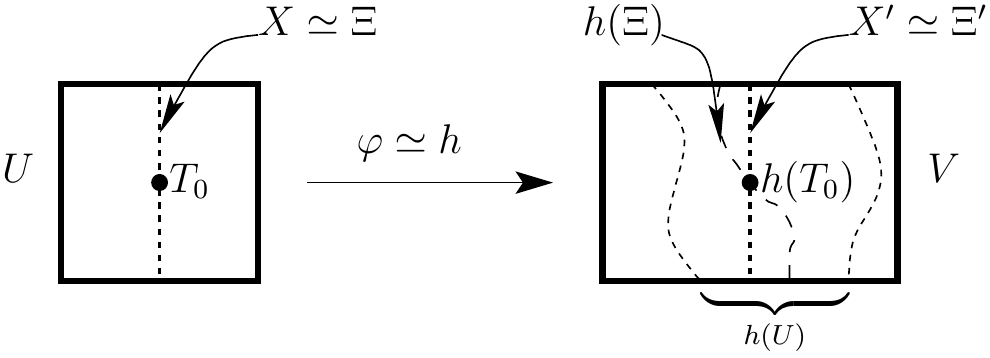}
 \end{center}
 \caption{The construction of $\Xi$ and $\Xi'$. The diameter of $U$ is at most $\delta$ and the diameter of $V$ is at most $\eps$.}
 \label{fig:voisinages}
\end{figure}

 The next step in the proof is to build a map $\Xi \rightarrow \Xi'$ which is induced by~$h$. The definition itself is straightforward: for each $T \in \Xi$, there is a unique $T' \in \Xi'$ which belongs to the same connected component of $V$ as $h(T)$.
 Define $h_\eps (T)$ to be this $T'$.
 It is clear that for all $T \in \Xi$, $h_\eps (T) = h(T) - s(T)$, where $s(T)$ is a translation vector of norm less than $\eps$ (since the diameter of $V$ is less than $\eps$).
 Two things need to be shown: first, that $h_\eps$ is indeed continuous; then, that it can be extended to a function on all $\Omega$ and not just on $\Xi$.

 To this end, first prove that $s$ is a continuous function on~$\Xi$. Let $\eta > 0$, which is assumed to be small with respect to the size of the tiles, and $T_1 \in \Xi$.
 By continuity of $h$, there is $\nu > 0$ such that for all $T \in \Xi$ which is $\nu$-close to $T_1$, the images by $h$ of $T$ and $T_1$ are $\eta/2$-close.
 It implies that there is an $x \in \R^d$, with $\nr{x} < \eta$, such that $h(T_0)$ and $h(T) - x$ match exactly around the origin.
 In addition, since $h_\eps (T) = h(T) - s(T)$ and $h_\eps (T_0) = h(T_0) - s(T_0)$ both belong to the same canonical transversal $\Xi'$, they also need to match exactly around the origin. 
 Therefore, \emph{when restricted to $B(0,\eps^{-1})$}, one has (the notation $=_r$ stands for ``agree around the origin up to radius~$r$''):
 \begin{align*}
  h(T_0) - s(T_0) & =_{1/\eps} h(T) - s(T) \quad \text{so} \\
  h(T_0) - s(T_0) & =_{1/\eps} h(T) - (x + s(T_0)),
 \end{align*}
 and since $\nr{x}$ and $s(T_0)$ are both smaller than $\eps$ (which is small compared to the size of tiles), the only way that the last two tilings can match around the origin is if $s(T) = s(T_0) + x$, which implies $\nr{s(T_0) - s(T)} < \eta$.
 Since $\eta$ could have been chosen arbitrarily small, it proves that $s$ is continuous on $\Xi$.

 This last fact provides an instant proof of the continuity of $h_\eps$ on $\Xi$: indeed, $h_\eps = h - s$, both $h$ and $s$ are continuous, and $\R^d$ acts continuously on $\Omega$ by translations.

 Let us now show that $h_\eps$ can be extended on $\Omega$. It is done by extending $s$.
 First, $\phi$ (induced by $h$) extends to a map defined on the closure of its domain (by continuity of $h$ on $\Omega$). It is still noted $\phi$:
 \[
  \phi : \overline{B(0,\delta')} \times X  \lra  \overline{B(0,\eps')} \times X'
 \]
 Then, for all $\xi \in X$, there is a map $\phi_\xi: x \mapsto \phi(x,\xi)$ defined on the closed ball of radius $\delta'$, which is a homeomorphism onto its image.
 By construction, its image contains $0$ in its interior.
 Let $x_\xi := \phi_\xi^{-1} (0)$. It is an inner point of the closed disk. Then, let us explicitly provide a homeomorphism $f_\xi$ from the closed disk to itself, which leaves the boundary invariant and sends $0$ to $x_\xi$.

 Let $a(y) = 1-\nr{y}/\delta'$ defined on $\bar B (0,\delta')$. Notice that it is $0$ on the boundary, and takes value $1$ at $0$.
 Now, define
 \begin{equation}\label{eq:f-xi}
  f_\xi (y) = y + a(y) x_\xi.
 \end{equation}
 It is clear that $f_\xi (0) = x_\xi$, and $f_\xi$ leaves the boundary invariant.
 Furthermore, $y \mapsto a(y) x_\xi$ is $k$-Lipschitz, with $k=\nr{x_\xi} / \delta' < 1$.
 Therefore, $f_\xi$ is one-to-one, and so it is a homeomorphism onto its image.
 
 Now, one can check that the norm of $f_\xi(y)$ is
 \begin{align*}
  \nr{f_\xi(y)} & \leq  \nr{y} + \nr{x_\xi} - \nr{y}\nr{x_\xi} / \delta' \\
      & = \delta' \left( \frac{\nr{y}}{\delta'} + \frac{\nr{x_\xi}}{\delta'} - \frac{\nr{y}}{\delta'}\frac{\nr{x_\xi}}{\delta'}  \right)
 \end{align*}
 and a simple study of the function $(\alpha,\beta) \mapsto \alpha + \beta - \alpha \beta$ on the domain $[0,1]^2$ shows that it is bounded by one.
 Therefore, the range of $f_\xi$ is included in the closed ball $\bar B(0,\delta')$, and it contains its boundary.
 By an argument of algebraic topology\footnote{Assume $w$ in the interior of the ball is not on the range of $f_\xi$.
 Then $f_\xi$ composed with $z \mapsto \delta' (z-w)/\nr{x-w}$ maps the closed ball onto the closed sphere continuously, which is a contradiction.}, $f_\xi$ is therefore onto, and is a homeomorphism of the closed ball.

 It is a quick check that the function $\xi \mapsto f_\xi$ is continuous on $\Xi$.
 Indeed, the continuity of $h_\eps$ implies that $\xi \mapsto (x_\xi,\xi) = h^{-1}\circ h_\eps (0,\xi)$ is a continuous function of $\xi$ (expressed in coordinates, the chart maps being implicit).
 In addition, the map $x \mapsto (\id(\cdot) + a(\cdot) x)$ is continuous 
 for the topology of uniform convergence on the space of functions on the ball.
 It gives continuity of $\xi \mapsto f_\xi$.

 Now, it is possible to extend the function $h_\eps$ on $U \simeq B(0,\delta') \times X$: in coordinates, it is defined as $h_\eps (y,\xi) = (f_\xi(y), \xi)$. It amounts to precomposing $h$ with $f_\xi$ on each leaf. It is clear by definition of the $f_\xi$ that this map restricts to $h_\eps$ on $\Xi$, as expected.
 The continuity of $\xi \mapsto f_\xi$ shows that $h_\eps$ is continuous on $U$.
 Finally, because of the boundary conditions put on $f_\xi$, the function $h_\eps$ can be extended continuously by $h$ on the complement of $U$.
 
 The last point to prove is the isotopy between $h_\eps$ and $h$. It can be done by replacing the definition of $f_\xi$ in Equation~\eqref{eq:f-xi} by the following one:
 \[
 f_{\xi,t} (y) = y + t a(y) x_\xi,
 \]
 for a parameter $t \in [0,1]$.
 Leaving the rest of the construction unchanged, the result is a family of functions $(h_{t\eps})_{t \in [0,1]}$. For $t=1$, the function is just $h_\eps$; for $t=0$, it is $h_0=h$. This provides the desired isotopy.
\end{proof}

\section{Proof of the theorems}

 \subsection{Preliminary lemmas}

 The function $h_\eps$ being a well defined homeomorphism $\Omega \ra \Omega'$, it is possible for all $T \in \Omega$ to define $h_{\eps,T}: \R^d \ra \R^d$ in the same way $h_T$ was defined.

\begin{lem}\label{lem:quasi-lipschitz}
 There is a $M > 0$ such that  for all $T \in \Xi$ and all $x \in \R^d$ satisfying $T-x \in \Xi$,
 \[
  \nr{h_{\eps,T} (x)} \leq M \nr{x}.
 \] 
\end{lem}

\begin{proof}
 Let $T \in \Xi$ and $T'=h_\eps(T) \in \Xi'$. Let $\ron D$ and $\ron D'$ be the pointing rules associated with transversal $\Xi$ and $\Xi'$ respectively.
 It immediate to check that $x \in \ron D (T)$ if and only if $h_{\eps,T} (x) \in \ron D' (T')$.
 
 Let $\ron T$ be a triangulation rule locally derived from $\ron D$, such that the set of vertices is exactly $\ron D$. A triangulation is a tiling where each tile is a $d$-simplex, and tiles meet face-to-face (the faces being lower-dimensional simplices).
 One well known way to do so is to use Delaunay triangulations of $\R^d$ with vertices in $\ron D (T)$.
 Such triangulations are dual to the Voronoi construction. In some (non-generic) cases, the Delaunay triangulation is not unique, but by making consistent choices, it is possible to ensure that $\ron T (T)$ is locally derived from~$T$. See~\cite{Sch93}, for a reference.
 
 Let $g:\R^d \ra \R^d$ be the function defined as follows:
 given $x \in \R^d$, it belongs to a simplex, so it is uniquely written as a convex combination (with non-negative coefficients) of its vertices, which are elements in $\ron D$:
 \[
  x = \lambda_0 v_0 + \lambda_1 v_1 + \ldots + \lambda_d v_{d}, \quad \lambda_k \geq 0, \quad \sum_{k=0}^d \lambda_k = 1.
 \]
 Then, define
 \[
  g(x) = \sum_{k=0}^{d}\lambda_k h_{\eps,T} (v_k).
 \]
 This map $g$ is onto, but does not need to be one-to-one, see Figure~\ref{fig:triang-wrong}. It is affine by parts, and using finite local complexity of $\ron T (T)$, there are finitely may linear maps underlying the affine maps. A consequence of this is that $g$ is Lipschitz, with coefficient $M$.
 Finally, $g$ agrees with $h_{\eps, T}$ on $\ron D$.
 
 \begin{figure}[ht]
 \begin{center}
  \includegraphics{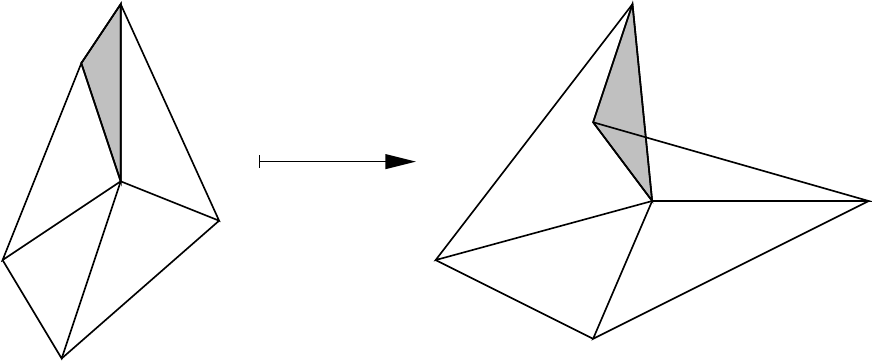}
 \end{center}
 \caption{A local pattern of the triangulation can be sent by $g$ to a pattern which is combinatorially a triangulation, but with an overlap in the geometric realization. Remark that the shaded triangle has a ``wrong'' orientation on the right. In this case, the map $g$ is not one-to-one.}
 \label{fig:triang-wrong}
\end{figure}

 Now, let us prove the lemma.
 Let $x \in \R^d$ such that $T-x \in \Xi$ (that is: let $x \in \ron D$).
 Then
 \begin{align*}
  \nr{h_{\eps,T}(x)} & = \nr{h_{\eps,T}(x)-h_{\eps,T}(0)} \\
    & = \nr{g(x) - g(0)} \\
    & \leq M \nr{x-0} = M\nr{x}.
 \end{align*}
\end{proof}

This lemma has the following consequence.
\begin{cor}\label{cor:lipschitz}
  Let $h: \Omega \lra \Omega'$ be a homeomorphism between two aperiodic repetitive FLC tiling spaces.
  There is a $M > 0$ and a $C > 0$ such that  for all $T \in \Omega$ and all $x \in \R^d$,
 \[
  \nr{h_{T} (x)} \leq M \nr{x} + C.
 \] 
\end{cor}

\begin{proof}
 Let $\Xi$ and $\Xi'$ be the vertical transversal involved in the construction of $h_\eps$, and let $\ron D$ be the local pointing rule associated with $\Xi$.
 Using the Delaunay property of $\ron D (T)$ for $T \in \Omega$, there is $R > 0$ such that any ball $B(x,R)$ intersects $\ron D (T)$, for any $T \in \Omega$.
 
 Using continuity of $(T,x) \mapsto h_{\eps,T}(x)$ (see Theorem~\ref{thm:hT}), the image of the compact set $\Omega \times \overline{B(0,R)}$ is compact, and in particular is bounded. It is therefore included in $B(0,c)$ for some~$c$.
 
 Let now be $x \in \R^d$ and $T \in \Omega$. Then there are $y,z \in \R^d$ of length less than $R$ such that $T-y \in \Xi$ and $T-x-z \in \Xi$. One can write:
 \begin{align*}
  \nr{h_{\eps,T}(x)} & = \nr{h_{\eps,T}(y + (x+z-y) -z)} \\
     & = \nr{h_{\eps,T}(y) + h_{\eps,T-y} (x+z-y) + h_{\eps,T-x-z}(-z)} \\
     & \leq c + M(\nr{x} + \nr{y} + \nr{z}) + c \\
     & \leq 2c +2MR +M\nr{x}.
 \end{align*}
 The equality from first line to second uses the definition of $h_{\eps,T}$, the first inequality follows from Lemma~\ref{lem:quasi-lipschitz}.
 In the end, since $\nr{h_T(x)-h_{\eps,T}(x)} \leq \eps$ for all $T$ and $x$, the result follows, with $C=2c+2MR + \eps$.
\end{proof}

These results are key for estimating the size of $h_T(B(0,R))$ as $R$ grows.
A few additional lemmas still need to be proved before tackling the main theorem.

The two following lemmas deal with the following problems: how can a given $x \in \R^d$ be approximated by a sum of small ``return vectors'' (Definition~\ref{def:return-vect}); and how $h_{\eps,T}(v)$ is locally constant in $T$ when $v$ is a small return vector.

\begin{lem}\label{lem:jump}
 Let $\ron D$ be a Delaunay set which is $R$-relatively dense, that is:
 \[
  \bigcup_{p \in \ron D} B(p,R) = \R^d.
 \] 
 Let $x \in \ron D$ and $y \in \R^d$, such that $\nr{x-y} \geq 2R$.
 Then there exists $x' \in \ron D$, such that
 \[
  \nr{x'-x} \leq 3R \quad \text{and} \quad \nr{x'-y} \leq \nr{x-y}-R.
 \]
\end{lem}

\begin{proof}
 Define
 \[
  z = x - 2R \frac{x-y}{\nr{x-y}}.
 \]
 By the Delaunay property, there exists $x' \in \ron D$ which is within distance $R$ of $z$.
 It is a quick check that $\nr{x-x'} \leq 3R$, and
 \begin{align*}
  \nr{x'-y} & \leq \nr{x'-z} + \nr{z-y} \\
            & \leq R + \nr{x-y}-2R  \\
            &  = \nr{x-y}-R.
 \end{align*}
\end{proof}

\begin{defin}\label{def:return-vect}
 Let $\Omega$ be a tiling space, and $\Xi$ be a canonical transversal associated with the local pointing rule $\ron D$.
 Then the set
 \[
  \ron V = \{v \in \R^d \tq \exists T \in \Xi, \ T-v \in \Xi \}
 \]
 is discrete (by FLC), and is called the set of return vectors on $\Xi$. By definition,
 \[
  \ron V = \bigcup_{T \in \Xi} \ron D (T).
 \]
 Given $v \in \ron V$, define
 \[
  \Xi_v = \{T \in \Xi \tq T-v \in \Xi \}.
 \]
 By finite local complexity, it is a clopen subset of $\Xi$.
\end{defin}

 \begin{lem}\label{lem:vect-retour}
  Let $h_\eps: \Omega \lra \Omega'$ be a homeomorphism, which sends a transversal $\Xi$ to a transversal $\Xi'$. Let $\ron V$ be the set of return vectors on $\Xi$.
  
  Let $\lambda > 0$, and $\ron V_\lambda$ be the set of all return vectors of length at most $\lambda$.
  Then there exists $\nu > 0$ such that for all $v \in \ron V_\lambda$ and all $T_1,T_2 \in \Xi_v$,
  \[
   D(T_1,T_2) \leq \nu \Longrightarrow h_{\eps,T_1}(v) = h_{\eps,T_2}(v).
  \]
 \end{lem}
 
 \begin{proof}
  Given $v \in \ron V_\lambda$, the map
  \[
   T \mapsto h_T (v)
  \]
  is continuous on $\Xi_v$, and it maps $v$ to return vectors on $\Xi'$ (indeed, by definition of $\Xi_v$, both $T$ and $T-v$ are in $\Xi$, and therefore their images are in $\Xi'$).
  It is a continuous map from a compact set to a discrete set (by finite local complexity). It has therefore finite range, and is locally constant. This gives a constant $\nu_v$, such that
  \[
   D(T_1,T_2) \leq \nu_v \Longrightarrow h_{\eps,T_1}(v) = h_{\eps,T_2}(v).
  \]
  Since $\ron V_\lambda$ is finite, the lemma follows.
 \end{proof}

 \subsection{Proof of Theorem~\ref{thm:main} on complexity}

 Let $\Xi$, $\Xi'$ and $h_\eps$ be given by Lemma~\ref{lem:redressement}. Denote by $\ron D$ and $\ron D'$ the local pointing rules associated with $\Xi$ and $\Xi'$.
 
 The strategy of this proof is to use Lemma~\ref{lem:quasi-lipschitz} to prove that whenever $T,T' \in \Xi$ disagree within radius $r$, then their images in $\Xi'$ need to disagree within radius at most $Mr$, for some constant $M$ which does not depend on $r$.
 It will prove that
  \[
   p_\Xi (r) \leq p_{\Xi'} (Mr).
  \]
 The other inequality will be obtained by repeating the same argument for $h_\eps^{-1}$.

 Let $R$ be the constant such that the Delaunay sets $\ron D (T)$ are $R$-relatively dense.
 Let $\nu$ be the constant given by Lemma~\ref{lem:vect-retour} for $\lambda=3R$, so that if $v \in \ron V_{3R}$ and $T_1,T_2 \in \Xi_v$ agree up to radius $\nu^{-1}$, then $h_{T_i}(v)$ does not depend on $i$.

 Let $\delta > 0$, which is assumed to be small with respect to the radius of the tiles, and smaller than $\nu$. Assume also that $\delta^{-1} > 3R$.
 By uniform continuity\footnote{In its contrapositive form.} of $h^{-1}$, there is $\eta > 0$ such that
 \[
  \forall T_1,T_2 \in \Xi, \ D(T_1,T_2) \geq (\delta^{-1}+3R)^{-1} \Rightarrow D(h(T_1),h(T_2)) \geq \eta.
 \]

 Let now $T_1,T_2 \in \Xi$, such that
 \[
  D(T_1,T_2) \geq (\delta^{-1} + r)^{-1},
 \]
 which means that these two tilings fail to agree on a ball of radius $r+\delta^{-1}$.
 Then, there exists $x \in \R^d$ with $\nr{x} \leq r$, such that $T_1-x$ and $T_2-x$ fail to agree on a ball of radius $\delta^{-1}$.
 The natural follow-up is to claim that $h(T_1)-h_{T_1}(x)$ and $h(T_2)-h_{T_2}(x)$ disagree within radius $\eta^{-1}$, and since $\nr{h_{T_i}(x)} \leq M \nr{x}$, then $h(T_1)$ and $h(T_2)$ disagree within radius $M r + \eta^{-1}$.
 This line of reasoning fails because there is no guarantee that $h_{T_1}(x) = h_{T_2}(x)$.
 Some precautions have to be taken.

 Run the following algorithm, which is illustrated by Figure~\ref{fig:algo}:
 \begin{enumerate}
  \item Initialize $k=0$, and start with $x_k=x_0=0$;
  \item If $T_1-x_k$ and $T_2-x_k$ do not agree on $B(0,3R+\delta^{-1})$, stop the algorithm.
    If they do agree (which means that $\nr{x-x_k}$ must be at least $3R$), continue to step~(3).
  \item Using Lemma~\ref{lem:jump}, there exists $x_{k+1} \in \ron D(T_i)$ such that $\nr{x_k-x_{k+1}} \leq 3R$, and $\nr{x-x_{k+1}} \leq \nr{x-x_k} -R$. Since $T_1-x_k$ and $T_2-x_k$ agree at least up to radius $3R$, the choice of $x_k$ can be made the same for $T_1$ and $T_2$.
  \item Return to step (ii).
 \end{enumerate}
 
 \begin{figure}
  \begin{center}
   \includegraphics{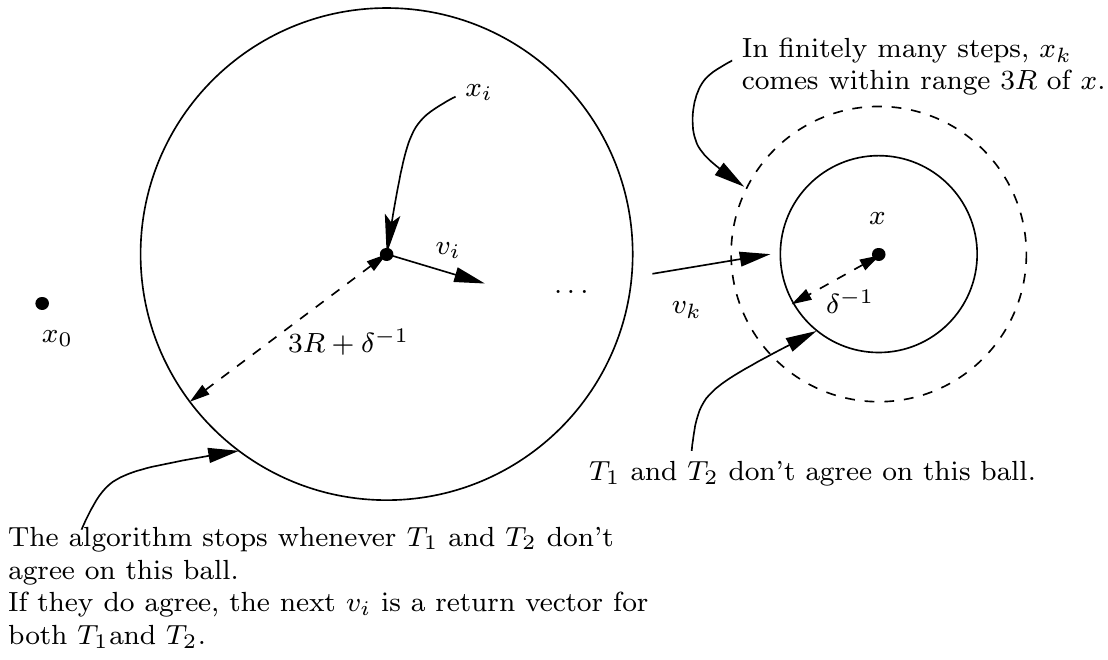}
  \end{center}
  \caption{An illustration of the algorithm.}
  \label{fig:algo}
 \end{figure}
 
 This algorithm terminates. Indeed, the distance between $x_k$ and $x$ is reduced at each step by at least $R$.
 After finitely many steps, $x_k$ and $x$ are within distance less than $3R$ of each other. When this happens, the stopping condition (2) applies, and the algorithm stops.
 The result of the algorithm is a finite sequence of vectors $(x_i)_{i=0,\ldots,k}$ such that each $v_i:=x_{i+1}-x_i$ is a return vector to $\Xi$ in $\ron V_{3R}$.
 Furthermore, for all $i<k$, $T_1-x_i$ and $T_2-x_i$ agree up to radius at least $\delta^{-1}$, which is greater than $\nu^{-1}$.
 Therefore, by Lemma~\ref{lem:vect-retour},
 \[
  h_{T_1-x_i} (v_i) = h_{T_2-x_i}(v_i),
 \]
 and these are return vectors to $\Xi'$, which we note $(w_i)_{i=0,\ldots,k-1}$.
 
 Because of the stopping condition of the algorithm above,
 \[
  D(T_1 - x_k, T_2 - x_k) \geq (3R+\delta^{-1})^{-1}, \text{ and therefore } D(h(T_1 - x_k), h(T_2 - x_k)) \geq \eta.
 \]
 Now, just remark that
 \begin{align*}
  h(T_i - x_k) & = h(T_i - v_0 - v_1 -\ldots -v_{k-1}) \\
     & = h(T_i) - w_0 - \ldots - w_{k-1}.
 \end{align*}
 In particular, there is a \emph{same} $y$ such that $h(T_i-x) = h(T_i) - y$, $i=1,2$.
 
 It was proved that $h(T_1)-y$ and $h(T_2)-y$ disagree on the ball of radius $\eta^{-1}$. Additionally, 
 \begin{align*}
  \nr{y} & \leq M \nr{x_k} \\
         & \leq M(\nr{x}+3R),
 \end{align*}
 where the first inequality if obtained by Lemma~\ref{lem:quasi-lipschitz}, and the second by the fact that $\nr{x-x_k} \leq 3R$.

 To sum up, it was proved that whenever two tilings $T_1,T_2 \in \Xi$ disagree within radius $\delta^{-1}+r$, then their images have to disagree within radius at most $M r +C$ for a constant $C= \eta^{-1}+3MR$.
 Therefore,
 \[
   p_\Xi (r + \delta^{-1}) \leq p_{\Xi'}(Mr+C).
 \]
 Up to changing $M$ and restricting $r$ to large values, one gets $p_\Xi (r) \leq p_{\Xi'} (Mr)$.
 
 The other inequality is obtained by reversing the roles of $h$ and $h^{-1}$.

 \subsection{Proof of Theorem~\ref{thm:repet} on repetitivity}
  
  The proof of this theorem is built on top of the proof of Theorem~\ref{thm:main}.
  The additional arguments are fairly simple, since repetitivity is very well translated in terms of recurrence and return vectors.
  How the norm of return vectors is changed by a homeomorphism $h$, in turn, is well evaluated by Lemma~\ref{lem:quasi-lipschitz}.

  Let $h:\Omega \lra \Omega'$ be a homeomorphism and $h_\eps$ be the approximation obtained from Lemma~\ref{lem:redressement}, which sends the canonical transversal $\Xi$ to the canonical transversal $\Xi'$.
  Let $\ron R$ and $\ron R'$ be the respective repetitivity functions of $\Omega$ and $\Omega'$ relative to $\Xi$ and $\Xi'$.
  Let also $M$ be the constant given by Lemma~\ref{lem:quasi-lipschitz}.
  
  The goal is to compute bounds on $\ron R(r)$ in terms of $\ron R'$.
  Let $P$ be a patch of radius $r$ of the form $T_0 \cap B(0,r)$ for $T_0 \in \Xi$.
  It defines a clopen set of $\Xi$ by
  \[
   U_{P} := \{ T \in \Xi \ ; \ T \cap B(0,r) = P  \}.
  \]
  We want to compute an upper bound for the the norm of a return vector from any tiling $T \in \Xi$ to $U_P$, in a way which is uniform in $P$.
  
  Let $V := h_\eps (U_{P})$. It is a clopen set, since $h_\eps$ is a homeomorphism.
  Now, consider the partition of $\Xi'$: $\{V_{P'} \}_{P'}$, where $P'$ ranges over all patches of the form $P'=\{T' \cap B(0,Mr) \ ; \ T' \in \Xi' \}$.
  It is claimed that any element of this partition is either completely included in $V$ or disjoint.
  Indeed, let $T'_1, T'_2 \in \Xi'$ be two tilings which agree up to radius $Mr$ (so that they belong to the same element of the partition defined above).
  It was proved in the proof of Theorem~\ref{thm:main} that whenever $T_1, T_2 \in \Xi$ disagree within radius $r$, then their images need to disagree within radius at most $Mr$.
  In a contrapositive form, it shows that $h_\eps^{-1}(T'_1)$ and $h_\eps^{-1}(T'_2)$ agree up to radius $r$. So they are either both in $U_P$, or none is.
  So either and $T'_1, T'_2$ are both in $V$, or none is. It results that
  \[
   h_\eps (U_P) = \bigsqcup_{i=1}^k V_{P'_i}, \quad \text{where } P'_i = T'_i \cap B(0,Mr) \text{ for some } T'_i \in \Xi'.
  \]
  
  To conclude, given $T \in \Xi$ let us find a return vector to $U_P$. Let $T':=h_\eps(T)$. By definition, there is $y$ be such that $T'-y \in V_{P'_i}$ and $\nr{y} \leq \ron R(Mr) + 1$.
  So $x := (h_\eps^{-1})_{T'}$ satisfies $T-x \in U_P$.
  To conclude, we use the Lipschitz-like estimation of Lemma~\ref{lem:quasi-lipschitz} on $(h_\eps^{-1})_{T'}$ to get:
  \[
    x \leq \tilde M (\ron R' (Mr) +1) + C.
  \]
  For any $T \in \Xi$ it is possible to obtain a return vector to $U_P$ in this way. It shows (up to renaming the constants and maybe restricting to $r$ big enough) that
  \[
   \ron R (r) \leq \lambda \ron R (Mr).
  \]
  The other inequality, again, is obtained by reversing the roles of $h$ and $h^{-1}$.

\section{Concluding remarks on deformations}
\label{sec:deformations}

There have been essentially two approaches for understanding tiling deformations.
In the first one---due to Clark and Sadun~\cite{CS06}---a deformation is defined as a map from the set of edges of a polytopal tiling to $\R^d$.
The image of an edge is a (a priori) different edge; fitting edges together, it defines a deformation of the tiling.
This map needs to be \emph{pattern-equivariant} in the sense that the image of an edge of $T$ should be determined by the local pattern around it.

The other point of view uses Kellendonk's pattern-equivariant differential forms~\cite{Kel03,Kel08}.
A function defined on $\R^d$ is $T$-equivariant for some fixed tiling $T$ if its value at $x$ only depends on the local pattern of $T$ around $x$. A deformation of a point pattern $\ron D (T)$ is then the image $f(\ron D (T))$ under a differentiable function $f: \R^d \ra \R^d$, with $T$-equivariant differential.

In both cases, a deformation is associated with an element in the cohomology of the tiling space with coefficients in $\R^d$: the \v{C}ech cohomology on the one hand (interpreted in~\cite{CS06} by using the pattern-equivariant formalism), and the pattern-equivariant De Rham cohomology on the other hand.
These cohomology groups turn out to agree~\cite{KP06,CS06,BK10}.



In the work of Clark and Sadun, a deformation of an aperiodic, repetitive, FLC tiling of $\R^d$ by polytopes $T$ is given by a map $f$ defined on the set of oriented edges of $T$, valued in $\R^d$, such that for some $R > 0$,
\begin{enumerate}
 \item $f(e)$ only depends on the local configuration of $T$ around $e$ up to radius~$R$;
 \item if $e_1, \ldots, e_n$ is a closed circuit of edges in $T$, then $\sum_{i=1}^n f(e_i) = 0$.
\end{enumerate}
The second condition is a cocycle condition. Such a function $f$ is a coboundary if there is $s$ defined on vertices of $T$, such that
\begin{enumerate}
 \item $s(v)$ only depends on $(T-v) \cap B(0,R)$;
 \item $f(e) = s(e^+) - s(e^-)$, where $e^+$ and $e^-$ denote the target and source vertices of the oriented edge~$e$.
\end{enumerate}
Then they prove that the quotient of cocycles by coboundaries is exactly isomorphic to $\check \Ho^1 (\Omega;\R^d)$.
Furthermore, such a cocycle defines a deformation of the tiling as follows: assuming (without loss of generality) that $0$ is a vertex of $T$, define
\[
 F (v) = \sum_i f(e_i),
\]
where $(e_i)_{i}$ is a path of edges from $0$ to the vertex $v$. By definition of $f$, the value of $F(v)$ does not depend on the path chosen. Then, at least if $f$ is close enough to the identity, $F(T)$ defines a tiling of $\R^d$ (a sufficient condition is that the image of a tile should be a non-degenerate tile with the same orientation).
Furthermore, there is then a homeomorphism between the tiling spaces of $T$ and $F(T)$.

In light of this presentation, it appears that any homeomorphism between tiling spaces is related to a deformation.
Let $h:\Omega \ra \Omega'$ be a homeomorphism, and $h_\eps, \Xi, \Xi'$ be the data provided by Lemma~\ref{lem:redressement}.
Let $\ron D$ and $\ron D'$ be the local pointing rules associated with $\Xi$ and $\Xi'$, and let $T_0 \in \Xi$.
Let $\ron T$ be the local derivation introduced in the proof of Lemma~\ref{lem:quasi-lipschitz}, such that for any $T \in \Xi$, $\ron T (T)$ is a triangulation with vertices $\ron D (T)$.
Let $\ron T'$ be the image rule, with vertices given by the pointing rule $\ron D'$, and such that two vertices are linked by an edge in $\ron T'$ if and only if their preimages by $h_\eps$ are linked by an edge in $\ron T$.
Even though image of the triangulation $\ron T$ does not need to be a triangulation itself (see Figure~\ref{fig:triang-wrong}), the map $\ron T (T) \mapsto \ron T'(h_\eps(T))$ is induced by a deformation in the sense above.

Let us make it explicit: let $e_0$ be an edge of the triangulation $\ron T (T_0)$. Say $e_0 = (v,e)$ where $v$ is a vertex of $\ron T (T_0)$ (the source of $e_0$), and $e$ is a vector of $\R^d$.
Define $\Xi_e := \Xi \cap (\Xi-e)$ (the sub-transversal of all tilings in $\Xi$ which have an edge $e$ at the origin, pointing away).
Then the map $T \mapsto h_{\eps,T} (e)$ is continuous, and its image is a return vector to $\Xi'$.
By Lemma~\ref{lem:vect-retour}, the image of $e$ under $h_{\eps,T}$ only depends on the configuration of $T$ around the origin up to a certain radius, say $R_e$.
Then, define $f(e_0) = h_{\eps,T_0-v}(e)$.
Defining $f$ similarly for other edges, and letting $R = \max \{R_e \}$, $f$ defines a deformation cocycle with radius $R$.
Then $f$ induces a deformation $F: \ron T (T_0) \mapsto \ron T' (h_\eps(T_0))$.
This function $F$ extends to $\Xi$, and it coincides with $h_\eps$ on this set.

It may be enlightening to adopt a groupoid point-of-view on these deformations.
Let $\ron G$ be the reduction on $\Xi$ of $\Omega \rtimes \R^d$; it consists of the elements with both range and source in $\Xi$. Elements of $\ron G$ are of the form $(T,v)$ where $T \in \Xi$ and $v$ is a return vector of $T$ on $\Xi$.
Then the deformation data in Clark--Sadun's presentation is a partially defined map $f:\ron G \ra \R^d$. It is defined only on the elements $(T,v)$ where $v$ is a return vectors corresponding to an edge in~$T$.
That being said, those elements generate the groupoid and $f$ satisfies the cocycle condition. Therefore, it extends to a groupoid morphism $\ron G \ra \R^d$.
There is an equivalence between deformations in the sense of Clark--Sadun and groupoid morphisms $\ron G \ra \R^d$ such that $(T,v) \mapsto f(T,v)$ is locally constant in the variable $T \in \Xi$.
It is now very apparent that $h_\eps$ induces a deformation, since it was proved that $(T,x) \mapsto h_{\eps,T}(x)$ is locally constant in $T$, if one restricts to $T \in \Xi$.

The function $F = \restr{(h_{\eps})}{\Xi}$ defines a class in $\check \Ho^1 (\Omega; \R^d)$. This cohomology group is isomorphic to Kellendonk's strongly equivariant cohomology group $\Ho^1_{\pe, \mathrm{s}} (\Omega; \R^d)$.

The proof of the following theorem will appear in a future paper, as well as a more thorough interpretation of deformations in terms of groupoids and groupoid cohomology.

\begin{thm*}
 Let $h: \Omega \lra \Omega'$ be a homeomorphism between two aperiodic, FLC, repetitive tiling spaces.
Then $h$ defines an element in Kellendonk's mixed cohomology group $\Ho^1_{\pe,\mathrm{m}}(\Omega;\R^d)$. This element is the image in the mixed group of the deformation cocycle defined above under the quotient map
\[
\Ho^1_{\pe,\mathrm{s}}(\Omega;\R^d)  \lra \Ho^1_{\pe,\mathrm{m}}(\Omega;\R^d)
\]
(see~\cite{Kel08}). This element in the quotient does not depend on the choice of $\Xi$, $\Xi'$ or $h_\eps$.

Besides, if $h_1: \Omega \ra \Omega'$ and $h_2: \Omega \ra \Omega''$ are homeomorphisms which define the same element in the mixed group, then $\Omega'$ and $\Omega''$ are topologically conjugate.
\end{thm*}

This theorem indicates that Kellendonk's mixed group $\Ho^1_{\pe,m}(\Omega;\R^d)$ seems to be an appropriate invariant for classifying homeomorphisms of $\Omega$ in another FLC tiling space modulo conjugacy.
This should be compared to recent work of Kellendonk and Sadun~\cite{KS12} in which they show that some elements of the group $\Ho^1_{\pe,s}$ (the so-called ``infinitesimals'') classify topological conjugacy of FLC tiling spaces modulo invertible local derivation.

Finally, it is legitimate to wonder whether a homeomorphism between tiling spaces can be perturbed in order to preserve \emph{all} transversals (or equivalently, to choose $\Xi$ in such a way that the map $g$ in Lemma~\ref{lem:quasi-lipschitz} is a homeomorphism).

\begin{conj*}
 Given a homeomorphism $h: \Omega \ra \Omega'$ between two aperiodic repetitive FLC tiling spaces, there is an homeomorphism $h'$, isotope to $h$, such that for \emph{any} canonical transversal $\Xi$ in $\Omega$, $h'(\Xi)$ is a canonical transversal in $\Omega'$.
\end{conj*}

\bibliographystyle{abbrv}

\bibliography{../../../biblio}{}


\end{document}

%% file: macros-arxiv.tex
\usepackage[T1]{fontenc}
\usepackage{lmodern}
\usepackage[utf8]{inputenc}
\usepackage[english]{babel}

\usepackage[all]{xy}
\usepackage{amsmath}
\usepackage{amssymb}
\usepackage{amsthm}

\usepackage{graphicx}
\usepackage{psfrag}

\usepackage[small]{caption}

\usepackage{enumerate}

\theoremstyle{plain}
\newtheorem{thm}{Theorem}[section]
\newtheorem*{thm*}{Theorem}
\newtheorem{prop}[thm]{Proposition}
\newtheorem*{prop*}{Proposition}
\newtheorem*{conj*}{Conjecture}

\theoremstyle{definition}
\newtheorem{defin}[thm]{Definition}
\newtheorem{cor}[thm]{Corollary}
\newtheorem*{cor*}{Corollary}
\newtheorem{lem}[thm]{Lemma}

\theoremstyle{remark}
\newtheorem{rem}[thm]{Remark}
\newtheorem*{rem*}{Remark}

\newtheorem*{notation*}{Notation}

\newtheoremstyle{citing}
  {\topsep}
  {\topsep}
  {\itshape}
  {}
  {\bfseries}
  {.}
  {.5em}
  {\thmnote{#3}}

\theoremstyle{citing}

\newcommand{\field}[1]{\mathbb{#1}}

\newcommand{\R}{\field{R}}

\newcommand{\N}{\field{N}}
\newcommand{\Z}{\field{Z}}

\newcommand{\ron}[1]{\mathcal{#1}}

\newcommand{\tq}{\,;\,}

\newcommand{\nr}[1]{\left\Vert #1\right\Vert}

\newcommand{\restr}[2]{\left.#1\right|_{#2}}

\DeclareMathOperator{\card}{Card}

\DeclareMathOperator{\orb}{Orb}
\DeclareMathOperator{\id}{id}

\newcommand{\ie}{{\emph{i.e.~}}}

\newcommand{\eps}{\varepsilon}
\renewcommand{\phi}{\varphi}
\renewcommand{\bar}[1]{\overline{#1}}

\newcommand{\ra}{\rightarrow}

\newcommand{\lra}{\longrightarrow}
